\documentclass[review,onefignum,onetabnum]{siamart171218}
\usepackage{amssymb}
\usepackage{titlesec}
\usepackage{amsmath,amscd}
\usepackage{amsfonts}
\usepackage{graphicx}
\usepackage{ulem}
\usepackage{subfigure}
\usepackage{amssymb}
\usepackage{float} 
\usepackage{multirow}
\usepackage{diagbox}
\usepackage{mathrsfs}
\usepackage{indentfirst}
\usepackage{geometry}
\usepackage{xcolor}
\usepackage{amssymb}
\usepackage{enumitem}
\usepackage{tikz}
\usepackage{etoolbox}
\usepackage{geometry}
\usepackage{fancyhdr}                                
\usepackage{lastpage}                                           
\usepackage{layout}
\usepackage{titlesec}
\newtheorem{assumption}{Assumption}[section]

\newcommand{\circled}[2][]{\tikz[baseline=(char.base)]
    {\node[shape = circle, draw, inner sep = 1pt]
    (char) {\phantom{\ifblank{#1}{#2}{#1}}};%
    \node at (char.center) {\makebox[0pt][c]{#2}};}}
\robustify{\circled}

\usepackage{lipsum}
\usepackage{amsfonts}
\usepackage{graphicx}
\usepackage{epstopdf}
\usepackage{algorithmic}
\ifpdf
  \DeclareGraphicsExtensions{.eps,.pdf,.png,.jpg}
\else
  \DeclareGraphicsExtensions{.eps}
\fi

\newsiamremark{remark}{Remark}
\newsiamremark{hypothesis}{Hypothesis}
\crefname{hypothesis}{Hypothesis}{Hypotheses}
\newsiamthm{claim}{Claim}

\headers{SVRG-PDFP for composite optimization}{Ya-Nan Zhu and~ Xiaoqun Zhang}

\title{A Stochastic variance reduced primal dual fixed point method for linearly constrained separable  optimization\thanks{Submitted to the editors DATE.}}

\author{Ya-Nan Zhu and Xiaoqun Zhang}

\usepackage{amsopn}

\ifpdf
\hypersetup{
  pdftitle={A Stochastic variance reduced primal dual fixed point method for linearly composite  optimization},
  pdfauthor={Ya-Nan Zhu and Xiaoqun Zhang}
}
\fi

\begin{document}
\nolinenumbers
\maketitle

\begin{abstract}
In this paper we combine the stochastic variance reduced gradient (SVRG) method \cite{SVRG} with the primal dual fixed point method (PDFP) proposed in \cite{PDFP} to solve a sum of two  convex functions and one of which is linearly composite. This type of problems are typically arisen in sparse signal and image reconstruction. The proposed SVRG-PDFP can be seen as a generalization of Prox-SVRG \cite{ProxSVRG} originally designed  for the minimization of a sum of two convex functions. Based on some standard assumptions, we propose two variants, one is for strongly convex objective function and the other is  for general convex cases. Convergence analysis shows  that the convergence rate of SVRG-PDFP is $\mathcal{O}(\frac{1}{k})$ (here $k$ is the iteration number) for general convex objective function and linear for strongly convex case. Numerical examples on machine learning and CT image reconstruction are provided to show the effectiveness of the algorithms. 
\end{abstract}


\section{Introduction.}
In  machine learning and imaging sciences, we often consider the following type of optimization problems: 
\begin{equation}\label{pb1}
\underset{x \in \mathbb{R}^d}{\min~} \frac{1}{n} \sum_{i = 1}^{n}f_i(x)+ (g \circ B) (x),   
\end{equation}
where $f_i:\mathbb{R}^d \to \mathbb{R} \cup \{ \infty\}$ is convex lower semi-continuous (l.s.c.) function with  Lipschitz continuous gradient, the function $g: \mathbb{R}^r \to \mathbb{R} \cup \{ \infty\}$ is also convex l.s.c. but may not be differentiable and $B : \mathbb{R}^{d} \rightarrow \mathbb{R}^{r}$ is a linear transform. 

In machine learning, the formulation \eqref{pb1} is known as regularized empirical minimization \cite{ERM} when $f(x)$ is some loss function defined on the data  and $g\circ B(x)$ is a regularizer. Generally the linear transform $B$ is set as identity in  many regularized empirical minimization problems. For example, the well-known Lasso problem takes the form: 
$f_i(x) = \frac{1}{2}(a_i^Tx - b_i)^2, i = 1, 2, \cdots,n$ here $b_i$ is the label of the  sample $a_i$, $x$ is the weight to be found and $g(\cdot)=\lambda \lVert \cdot \rVert_1$ for $\lambda > 0$. In binary classification task, $f_i(x)$ is replaced by logistic loss $f_i(x)=\log(1 + \exp(-b_ia_i^Tx))$ where $b_i \in \{-1,1\}$. To further improve the generalization ability of learning models, non-identity linear operator $B$ for the regularization has been considered in the literature.  For example, if we consider  $B = [G;I]$ where $G$ is determined by sparse inverse covariance selection \cite{Banerjee2008model}, then the problem is known as {graph-guided fussed Lasso} \cite{Kim}. Because of the large size of $G$ and $n$, it is necessary to design an algorithm with relatively simple iteration rule to solve this type of optimization problems. 

In  imaging science, this formulation is typically considered for solving an ill-posed inverse problem. For example,  tomographic problems consist of estimating a two or three dimensional function from a set of line integrals. Typically, a regularized  reconstruction model can be formulated as 
\begin{equation}\label{TVPb1}
\underset{x \in \mathbb{R}^d}{\min}~ \lVert \mathcal{A}x - f\rVert_2^2 + g(\nabla x)  
\end{equation}
where $x$ is the image to be reconstructed, $\mathcal{A}$  is Radon transform, $f$ is the measured projection vector and $\nabla(\cdot)$ is a discrete gradient operator. For $g$, if we choose $g = \lVert \cdot \rVert_1$, then \textbf{(\ref{TVPb1})} becomes the Total Variation-$L_2$ model (TV-$L_2$). In practice there are thousands or even millions of projections  and the operator $\nabla(\cdot)$ is necessary to ensure the quality of the reconstructed image. Suppose the number of projections is $n$, denote $\mathcal{A}_i$ as the $i$-th projection operator and $f_i$ is the $i$-th component of $f$, then the problem \textbf{(\ref{TVPb1})} can be reformulated as 
\begin{equation}\label{TVPb2}
\underset{x \in \mathbb{R}^d}{\min}~ \frac{1}{n}\sum_{i = 1}^{n}(\mathcal{A}_ix - f_i)^2 + \frac{1}{n}g(\nabla x)  
\end{equation}
which is in the form of \textbf{(\ref{pb1})}. 
 
 In this paper, we aim to consider a stochastic algorithm to solve the problems (\ref{pb1}) when the data size becomes large. For problems with a simple regularizer (i.e. $B = I$), one of the most popular deterministic  method is the class of  proximal gradient decent (PGD) (also known as Proximal Forward-backward  splitting method)  \cite{PFBS,FISTA,FastPPT,Nesterov,GFB} and there stochastic versions \cite{FOBOS,converPGD}. Denote $f(x)=\frac{1}{n} \sum_{i = 1}^{n}f_i(x)$, specifically in stochastic gradient method (SGD), one uses  a small portion of data to compute a noisy gradient, i.e. the stochastic gradient $\nabla \hat{f}(x)$ is computed as: 
\begin{equation}\label{SG}
\nabla \hat{f}(x) = \frac{1}{b}\sum_{i \in I_k} \nabla f_i(x), \qquad k = 1,\cdots,\frac{n}{b}.     
\end{equation} 
where $I_1,I_2,\cdots,I_{\frac{n}{b}}$ denote a disjoint partition of the index set $\{1,2,\cdots,n\}$ and the number of element in each $I_i,i = 1,\cdots,\frac{n}{b}$ is $b$ (which is known as batch size). Generally $I_k$ in \textbf{(\ref{SG})} is chosen randomly at each iteration. The idea of Prox-SG  method combines the stochastic gradient step (\ref{SG}) and a proximal iteration of $g$, which will reduce the computation cost from $\mathcal{O}(n)$ to $\mathcal{O}(b)$ at each iteration and generally the batch size $b \ll n$. Owing to the variance caused by random sampling, Prox-SG uses a diminishing step size rule which leads to a sub-linear convergence rate. In order to accelerate the convergence, the variance reduction technique was firstly considered for $g(\cdot) = 0$ \cite{SVRG,SAG}. For example, Le Roux et al. \cite{SAG} proposed  stochastic averaged gradient (SAG)  and Johnson and Zhang \cite{SVRG} developed another algorithm called stochastic variance reduced gradient (SVRG). Combining with PGD, Xiao and Zhang \cite{ProxSVRG} proposed the Prox-SVRG for solving the problems (\ref{pb1}) with $B=I$. In contrast to one-level stochastic gradient (\ref{SG}),  the idea of Prox-SVRG proceeds in two stages. First, the full gradient of the past estimate of $\tilde{x}$ is computed as $z = \frac{1}{n}\sum_{i = 1}^{n}\nabla f_i(\tilde{x})$. Then  an approximate gradient is computed by
\begin{equation}
 \nabla \hat{f}(x_k) = \frac{1}{b}\sum_{i \in I_k}(\nabla f_i(x_k) - \nabla f_i(\tilde{x})) + z,      
\end{equation}  
where $\tilde{x}$ is outer iterate and $x_k$ is inner iterate. Based on this modification, Prox-SVRG allows to use a constant step size and can achieve linear convergence for strongly convex objective function.


When the linear transform $B \not = I$, PGD type methods need to solve $\mathrm{Prox}_{g \circ B}(\cdot)$ which is not easy for many problems. In the deterministic setting, many algorithms such as split Bregman \cite{DR1,DR2} (or alternating direction of multipliers method (ADMM) \cite{ADMM1,ADMM2}), primal dual hybrid gradient (PDHG) \cite{PDHG}, fixed point method based on proximity operator ($\mathrm{FP^2O}$) \cite{FP2O}, primal dual fixed point method (PDFP) \cite{PDFP} are proposed and largely applied in imaging and data sciences. The ADMM-type method can be interpreted as a primal dual method solving the reformulation of problem \textbf{(\ref{pb1})} as follows: 
\begin{equation}
\begin{aligned}
&\min ~ f(x) + g(y) \\
& \quad s.t.\quad  Bx = y.
\end{aligned} 		
\end{equation} 	
and the method is proceeded by alternating updating the augmented Lagrangian: 
\begin{equation}
L(x,y,\lambda) = f(x) + g(y) + \langle \lambda,Bx - y \rangle + \frac{\rho}{2}\lVert Bx - y \rVert_2^2.	
\end{equation}
Different from ADMM-type methods, both PDFP and PDHG  solve the min-max reformulation of the problem \textbf{(\ref{pb1})}:
\begin{equation}
\underset{x \in \mathbb{R}^d,v \in V}{ \min\max~}f(x) + \langle Bx,v \rangle - g^*(v),
\end{equation}
where $g^*(\cdot)$ is the conjugate function of $g(\cdot)$ (see definition \textbf{\ref{defcon}}) and $V$ is the domain of $g^*(\cdot)$. The methods PDFP and ADMM are not the same in general. It was shown in \cite{PDFP} that the main advantage of PDFP over ADMM is to avoid subproblem solving and a simpler rule of parameter choosing.  

In the literature, there are many stochastic variants of ADMM, mainly for solving machine learning problems. For example, Stochastic ADMM (STOC-ADMM) \cite{STOCADMM}, Regularized Dual Averaging ADMM (RDA-ADMM) \cite{RDAOPGADMM}, Online Proximal Gradient ADMM (OPG-ADMM) \cite{RDAOPGADMM}, Stochastic Averaged Gradient ADMM (SA-ADMM) \cite{SAGADMM}, Scalable ADMM (SCAS-ADMM) \cite{SCASADMM}, Stochastic Dual Coordinate Ascent ADMM (SDCA-ADMM) \cite{SDCAADMM} and Stochastic Variance Reduced ADMM (SVRG-ADMM) \cite{SVRGADMM} etc. For PDHG, Stochastic Primal Dual Hybrid Gradient (SPDHG) was also proposed in \cite{SPDHG} for image reconstruction problems where in each iteration a subset of dual variable is randomly updated. 

 Recently, we proposed a stochastic PDFP (SPDFP) algorithm in \cite{ZZ20} for solving composite  problems (\ref{pb1}). Based on a strong convexity and some standard  assumptions on the gradient of $f(x)$, we established the convergence rate of SPDFP as $O(1/k^\alpha)$ with stepsize $\gamma_k=1/k^\alpha$,  where $k$ is the iteration number.  In this paper, we aim to improve the convergence order of the stochastic PDFP algorithm by considering a variance reduced  PDFP algorithm. The idea of the proposed algorithm  SVRG-PDFP apply SVRG to the gradient of $f(x)$.  Theoretically it can be shown that  the proposed algorithm  can achieve linear convergence rate for strongly convex case and $\mathcal{O}(1/k)$ for a general convex case.  Moreover it can be shown that for the special  case $B = I$, SVRG-PDFP reduces to  Prox-SVRG, thus SVRG-PDFP can be seen as a natural generalization of Prox-SVRG.  Finally, a byproduct of the convergence analysis is that when we use a full batch size $b=n$, the algorithm reduces to a determintic PDFP for generally convex function,  and we obtain   $o(1/k)$ convergence rate for  PDFP, that was not studied in the original work \cite{PDFP}. Finally, the numerical results are performed on graphic Lasso problem and 2D/3D CT reconstruction. The performance of SVRG-PDFP is illustrated with a detail comparison to PDFP, SPDFP and some variants of ADMM. In addition, the numerical results show that for largely scale image reconstruction problem stochastic algorithms are more beneficial in the case of limited  GPU computation resource.

The organization of the paper is as following. In the next section, SVRG-PDFP for strongly convex (Algorithm 1) and general convex case (Algorithm 2) will be present respectively. Then the convergence results of the algorithms will be provided with the details present in  Appendix. Finally, numerical experiments on graphic Lasso and CT image recontruction are  present with detailed  comparison to the other algorithms.

\section{Algorithm.}
In this section, we introduce the SVRG-PDFP for strongly convex and general convex cases respectively.  First we give the definition of $\mathrm{Prox}$ operator.
\begin{definition}\label{proxf}
The operator $\mathrm{Prox}_{g}(\cdot):\mathbb{R}^r \rightarrow \mathbb{R}^r $ is defined by
\begin{equation}
\begin{aligned}
\mathrm{Prox}_{g}(y): ~
& y \rightarrow \underset{x \in \mathbb{R}^r}{\arg\min}\big\{ f(x) + \frac{1}{2} \lVert x - y \rVert_2^2 \big\}.
\end{aligned}  
\end{equation}
\end{definition}

\begin{definition}\label{defcon}
{
The conjugate function of $g(\cdot)$ at $v$ is defined by
\begin{equation}
 g^*(v) = \underset{y \in dom(g)}{\sup} v^Ty - g(y)  
\end{equation}
where $v ~ \in ~ dom(g^*(\cdot)) = \{ v | g^*(v) < \infty \} = V^*$.
}
\end{definition}

Recall the primal dual fixed point method
\hspace*{\fill} \\
\begin{center}
\fbox{
\shortstack[l]{
    \textbf{Algorithm: Primal dual fixed point method} \\
    \rule{300pt}{2pt}\\
    Step 1: set $x_1 \in \mathbb{R}^d,v_1 \in \mathbb{R}^{m}$ and choose proper $\gamma>0,\lambda > 0$,\\
    Step 2: for $k = 1,2,\cdots$ \\
    \qquad\qquad $x_{k + \frac{1}{2}} = x_k - \gamma\nabla f(x_k)$ \\
    \qquad\qquad$v_{k + 1} = \big(I - \mathrm{Prox}_{\frac{\gamma}{\lambda}g}\big)\big(  Bx_{k + \frac{1}{2}} + (I - \lambda B B^T)v_k \big)$ \\
    \qquad\qquad $x_{k + 1} = x_{k + \frac{1}{2}} - \lambda B^T v_{k+1}$\\
    until the stop criterion is satisfied.
}
}
\end{center}
\hspace*{\fill} \\
PDFP can be reformulated as 
\begin{equation}\label{eq14}
\left\{
\begin{aligned}
y_{k + 1} & = x_k - \gamma \nabla f(x_k) - \gamma B^T v_k \\
v_{k + 1} & = \mathrm{Prox}_{\frac{\lambda}{\gamma}g^*}(\frac{\lambda}{\gamma}By_{k + 1} + v_k) \\
x_{k + 1} & = x_k - \gamma\nabla f(x_k) - \gamma B^Tv_{k + 1}.
\end{aligned}
\right.     
\end{equation}
One may refer to \cite{PDFP, ZZ20} for more details.  By combining the idea of SVRG and the reformulation of PDFP in \textbf{(\ref{eq14})}, we propose the following two algorithms for strongly convex and generally convex cases respectively. 

\begin{center}
\fbox{
\shortstack[l]{
    \textbf{Algorithm $1$: SVRG-PDFP for strongly convex problems} \\
	\rule{350pt}{2pt}\\
    \textbf{Input}: Choose proper $\gamma > 0,\lambda > 0, m > 0$, input $\tilde{x}_0 \in \mathbb{R}^d,\tilde{v}_0 \in \mathbb{R}^{r}$, batch size $b$. \\
    \textbf{for} $s = 0,1,2,\cdots$ do \\
    \qquad\quad $\tilde{x} = \tilde{x}_{s}$ \\
    \qquad\quad $x_0 = \tilde{x}_{s}, v_0 = \tilde{v}_{s}$ \\
    \qquad\quad $\tilde{z} = \nabla f(\tilde{x}) = \frac{1}{n}\sum_{i = 1}^n \nabla f_i(\tilde{x})$ \\
    \qquad\quad \textbf{for} $k = 0,1,2,\cdots, m -1$ do \\
    \qquad\quad ~~~~ Randomly choose $I_k \in \{I_1,I_2\cdots I_{\frac{n}{b}} \}$. \\
    \qquad\quad ~~~~ $\nabla \hat{f}(x_k) = \frac{1}{b} \sum_{i_k \in I_k} (\nabla f_{i_k}(x_k) - \nabla f_{i_k}(\tilde{x})) + \tilde{z}$ \\
    \qquad\quad ~~~~  $y_{k + 1} = x_k - \gamma \nabla \hat{f}(x_k) - \gamma B^Tv_k$ \\
    \qquad\quad ~~~~ $v_{k + 1} = \mathrm{Prox}_{\frac{\lambda}{\gamma}g^*}\big( \frac{\lambda}{\gamma} By_{k + 1} + v_k \big)$ \\
    \qquad\quad ~~~~  $x_{k + 1} = x_k - \gamma \nabla \hat{f}(x_k) - \gamma B^Tv_{k + 1}$\\
    \qquad\quad \textbf{end for} \\
    \qquad\quad $\tilde{x}_{s + 1} = \frac{1}{m}\sum_{i = 1}^{m} x_i$, $\tilde{v}_{s + 1} = \frac{1}{m}\sum_{i = 1}^{m} v_i$ \\
    \textbf{end for} \\
    \textbf{Output:} $\tilde{x}_s$.
}
}
\end{center}
\hspace*{\fill} \\
\begin{center}
\fbox{
\shortstack[l]{
	\textbf{Algorithm $2$: SVRG-PDFP for general convex problems} \\
    \rule{350pt}{2pt}\\
    \textbf{Input}: Choose proper $\gamma > 0,\lambda > 0, m > 0$, input $\tilde{x}_0 \in \mathbb{R}^d,\tilde{v}_0 \in \mathbb{R}^{r}$, batch size $b$. \\
    \textbf{for} $s = 0,1,2,\cdots,T - 1$ do \\
    \qquad\quad $\tilde{x} = \tilde{x}_{s}$ \\
    \qquad\quad $x_0 = \hat{x}_{s}, v_0 = \hat{v}_{s}$ \\
    \qquad\quad $\tilde{z} = \nabla f(\tilde{x}) = \frac{1}{n}\sum_{i = 1}^n \nabla f_i(\tilde{x})$ \\
    \qquad\quad \textbf{for} $k = 0,1,2,\cdots, m -1$ do \\
    \qquad\quad ~~~~ Randomly choose $I_k \in \{I_1,I_2\cdots I_{\frac{n}{b}} \}$. \\
    \qquad\quad ~~~~ $\nabla \hat{f}(x_k) = \frac{1}{b} \sum_{i_k \in I_k} (\nabla f_{i_t}(x_k) - \nabla f_{i_t}(\tilde{x})) + \tilde{z}$ \\
    \qquad\quad ~~~~  $y_{k + 1} = x_k - \gamma \nabla \hat{f}(x_k) - \gamma B^Tv_k$ \\
    \qquad\quad ~~~~ $v_{k + 1} = \mathrm{Prox}_{\frac{\lambda}{\gamma}g^*}\big( \frac{\lambda}{\gamma} By_{k + 1} + v_k \big)$ \\
    \qquad\quad ~~~~  $x_{k + 1} = x_k - \gamma \nabla \hat{f}(x_k) - \gamma B^Tv_{k + 1}$\\
    \qquad\quad \textbf{end for} \\
    \qquad\quad $\tilde{x}_{s + 1}= \frac{1}{m}\sum_{i = 1}^{m} x_i$, $\tilde{v}_{s + 1} = \frac{1}{m}\sum_{i = 1}^{m} v_i$, $\hat{x}_{s + 1} = x_m,\hat{v}_{s + 1} = v_m$ \\
    \textbf{end for} \\
    \textbf{Output:} $\overline{x}_T = \frac{1}{T}\sum_{i = 1}^T\tilde{x}_i$.
}
}
\end{center}

\section{Convergence Analysis.}
In this section, we present the convergence results of SVRG-PDFP. The proof can be found in Appendix.  First we present some useful definitions and assumptions.
\begin{definition}\label{def3}
The Bregman distance of a convex function $f$ is defined by
\begin{equation}
D_f(x,y) = f(x) - f(y) - \nabla f(y)^T(x - y).	
\end{equation}
\end{definition}

\begin{assumption}\label{assmp1}
The function $f_i(x),i = 1,2,\cdots,n$ is proper convex l.s.c. with $L_i$-Lipschitz continuous gradient i.e.
\begin{equation}
f_i(y) \leq f_i(x) + \nabla f_i(x)^T(y - x) + \frac{L_i}{2} \lVert y - x \rVert_2^2 \qquad \forall~ x,y \in dom(f_i).
\end{equation}
then it can be seen that the function $f(x)=\frac{1}{n}\sum_{i=1}^n f_i(x) $ also has Lipschitz continuous gradient and we denote its Lipschitz parameter as $\frac{1}{\beta}$.
\end{assumption}

\begin{assumption}\label{assmp2}
The function $f(x)$ is $\mu_f$-strongly convex i.e.
\begin{equation}
f(y) \geq f(x) + \nabla f(x)^T(y - x) + \frac{\mu_f}{2} \lVert y - x \rVert_2^2 \qquad \forall~ x,y \in dom(f).  
\end{equation}
\end{assumption}
\noindent Recall the problem \textbf{(\ref{pb1})}
\begin{equation}\label{eq20}
 \underset{x \in \mathbb{R}^d}{\min~} f(x) + (g \circ B) (x),  
\end{equation}
and denote $g^*(x)$ as the conjugate of $g(x)$. Define  $L(x,v) = f(x) + \langle Bx,v \rangle - g^*(v)$, then the min-max reformulation of \textbf{(\ref{eq20})} is given as follows
\begin{equation}\label{eq21}
\begin{aligned}
(x^*, v^*) & =  \underset{\quad ~x \in \mathbb{R}^d,v \in V}{\arg \min\max~} L(x,v) 
 = \underset{\quad ~ x \in \mathbb{R}^d,v \in V}{\arg \min\max~}f(x) + \langle Bx,v \rangle - g^*(v),
\end{aligned}
\end{equation}
where $V$ denotes the domain of $g^*(\cdot)$. The optimality condition of \textbf{(\ref{eq21})} is
\begin{equation}\label{22}
\left\{
\begin{aligned}
& f(x) - f(x^*) + (B^Tv^*)^T(x - x^*) \geq 0, \qquad x \in \mathbb{R}^d, \\
& g^*(v) - g^*(v^*) - (Bx^*)^T(v - v^*) \geq 0, \qquad v \in V.
\end{aligned}   
\right.
\end{equation}
where $(x^*, v^*)$ is an optimal primal dual solution pair.  The convergence is established w.r.t. $R(x,v)$:
\begin{equation}
\begin{aligned}
R(x,v) & = f(x) - f(x^*) - \nabla f(x^*)^T(x - x^*) + g^*(v) - g^*(v^*) - (Bx^*)^T(v - v^*)\\
& = D_f(x,x^*) + D_{g^*}(v,v^*)	
\end{aligned}
\end{equation}
where  $D_f(\cdot,\cdot),D_{g^*}(\cdot,\cdot)$ denote the Bregman distance of $f$ and $g^*$ respectively.  
\begin{proposition}
$R(x,v) \geq 0, \forall x \in dom(f),v \in V$.
\end{proposition} 
\begin{lemma}\label{lm1}
Suppose Assumption \textbf{\ref{assmp1}} holds, then the variance of $\nabla \hat{f}(x_k)$ is bounded by
\begin{equation}\label{lm1eq}
\begin{aligned}
& \mathbb{E}\big(\lVert \nabla \hat{f}(x_k) - \nabla f(x_k) \rVert^2 \big)
 \leq C(b)\big( D_f(x_k,x^*) + D_f(\tilde{x},x^*)\big),
\end{aligned}   
\end{equation}
where $C(b) = \frac{4(n - b)L_{max}}{b(n - 1)}$ and $L_{max} = \max\{L_1,\cdots,L_n\}$.
\end{lemma}

 Lemma \textbf{\ref{lm1}} gives the estimate of the variance of the stochastic gradient $\nabla \hat{f}(x_k)$. The proof can be found in the paper \cite{SVRGADMM} and we  also present in  Appendix for the  completeness of the paper.

\subsection{Convergence for Algorithm 1} 
\begin{theorem}\label{thm1}
Suppose Assumptions \textbf{\ref{assmp1}} and \textbf{\ref{assmp2}} hold,  and $g^*(v)$ is $\mu_{g^*}$-strongly convex. Let
\begin{equation}\label{thmrate}
\kappa =  \frac{1}{\mu_f\gamma(1 - \gamma M)m} + \frac{(m + 1) \gamma M}{(1 - \gamma M)m} + \frac{\gamma(1 - \rho_{min}(BB^T))}{\lambda\mu_{g^*}(1 - \gamma M)m} , 
\end{equation}
where $M = 4L_{max}C(b)$. Choose ${0 < \gamma \leq \min\{ \beta,\frac{1}{M}\}}$, ${0 <\lambda \leq \frac{1}{\rho_{max}(BB^T)}}$ ($\rho_{max}(\cdot),\rho_{min}(\cdot)$ denotes the maximum and minimum eigenvalue of a given matrix) and $m$ such that $\kappa < 1$, we then have
\begin{equation}
\mathbb{E}(R(\tilde{x}_s,\tilde{v}_s)) \leq \kappa^s R(\tilde{x}_0,\tilde{v}_0).
\end{equation}
Here $\tilde{x}_s,\tilde{v}_s$ denote the outer iterate of Algorithm $1$.
\end{theorem}

In Theorem \textbf{\ref{thm1}}, we require strong convexity of $g^*(\cdot)$. This may not be the case for some real applications, for example  $g(x) = \lVert \cdot \rVert_1$. In this case, we can use its Moreau-Yosida smoothing Huber norm to get an approximate solution. The Huber smoothing norm is given as follows: 
\begin{equation}
\lVert x \rVert_{\alpha} = \left \{
\begin{aligned}
& \frac{x_j^2}{2\alpha}, \qquad |x_j| \leq \alpha \\
& |x_j| - \frac{\alpha}{2}, \qquad |x_j| > \alpha \\
& j = 1,\cdots,d
\end{aligned}   
\right.
\end{equation}
where $x_j$ is the $j$-th component of vector $x$ and $\alpha > 0$. \\
\indent Moreover if $BB^T = I$ where $I$ is then identity, the strong convexity assumption of $g^*(\cdot)$ can be omitted. The property is stated in the  following  Corollary \textbf{\ref{thm11}}.
\begin{corollary}\label{thm11}
Suppose Assumption \textbf{\ref{assmp1}} and \textbf{\ref{assmp2}} hold and  $BB^T = I$.  Let
\begin{equation}\label{thm11rate}
\kappa =  \frac{1}{\mu_f\gamma(1 - \gamma M)m} + \frac{(m + 1)\gamma M}{(1 - \gamma M)m}.
\end{equation}
Choose $0 < \gamma < \min\{ \beta,\frac{1}{M}\}$, $\lambda  = 1$ and $m$ such that $\kappa < 1$, we have
\begin{equation}
\mathbb{E}(R(\tilde{x}_s,\tilde{v}_s)) \leq \kappa^s R(\tilde{x}_0,\tilde{v}_0).
\end{equation}
\end{corollary}
\begin{remark}\label{rem:}
\textbf{(Connection to Prox-SVRG \cite{ProxSVRG})}
If $B = I$,  $b = 1$ and $\lambda = 1$ in Algorithm 1, then SVRG-PDFP becomes Prox-SVRG and the convergence order and the  constant $\kappa$ in Theorem \textbf{\ref{thm11}} coincide with  the result in \cite{ProxSVRG}.
\end{remark}

\begin{remark}
(\textbf{Comparisons with SVRG-ADMM \cite{SVRGADMM}}) 
The linear convergence rate of SVRG-ADMM \cite{SVRGADMM} requires the assumption of  full row rank of $B$ but without strong convexity of $g^*(\cdot)$, compared to  SVRG-PDFP. And  the convergence constant $\kappa$ of SVRG-ADMM is 
\begin{equation}\label{ADMMrate}
\kappa =  \frac{1}{\mu_f\gamma(1 - \gamma M)m} + \frac{(m + 1) \gamma M}{(1 - \gamma M)m} + \frac{1}{\rho(1 - \gamma M)\rho_{min}(BB^T)m}
\end{equation}
where $\rho$ is the parameter on the augmented Lagrangian term. It can be seen that the only difference to the constant of SVRG-PDFP \textbf{(\ref{thmrate})} is  on the third term.
\end{remark}

\begin{remark}
(\textbf{Comparisons with SPDHG} \cite{SPDHG}) The conditions on the linear convergence rate of SPDHG also require  the strong convexity of both $f$ and $g^*$， which is the same as SVRG-PDFP. The advantage of SPDHG is that it does not need the Lipschitz continuous gradient of $f$, however the update of the primal variable  requires to solve the $\mathrm{Prox}$ operator of $f$. Thus if the computation of the $\mathrm{Prox}$ operator of $f$ is not easy and $f$ has Lipschitz continuous gradient,  SVRG-PDFP can be served as a good alternative.
\end{remark}

\subsection{Convergence of Algorithm 2}
\begin{theorem}\label{thm2}
Suppose Assumption \textbf{\ref{assmp1}} holds. Choose ${0 < \gamma \leq \min\{ \beta,\frac{1}{2M}\}}$ and ${0 < \lambda \leq \frac{1}{\rho_{max}(BB^T)}}$, we have
\begin{equation}\label{thm2eq}
\begin{aligned}
 \mathbb{E}(R(\overline{x}_T,\overline{v}_T)) 
& \leq \frac{\gamma ME}{(1 - 2\gamma M)mT},
\end{aligned}   
\end{equation}
where $\overline{x}_T$ is defined in Algorithm $2$, $\overline{v}_T = \frac{1}{T}\sum_{i = 1}^T\tilde{v}_i$ and  $E$ is a constant related to the initial point.
\end{theorem}

\begin{corollary}
If we let $b = n$, SVRG-PDFP reduces to PDFP and the ergodic convergence rate of PDFP is $\mathcal{O}(1/k)$.
\end{corollary}

\begin{remark}
The convergence rate of SVRG-PDFP is the same as SVRG-ADMM and SPDHG for the general convex function.
\end{remark}

Table \textbf{\ref{cgsum}} summarizes the convergence results of SVRG-PDFP, SVRG-ADMM\cite{SVRGADMM}, SPDHG\cite{SPDHG} based on the following conditions:
\begin{itemize}
    \item Strong convexity of $f(x)$ (S.C. $f_(x)$).
    \item Strong convexity of $g^*(x)$ (S.C. $g^*(x)$).
    \item Full row rank of matrix $B$ (FrkB).
    \item Lipschitz continuous gradient of $f(x)$ (Lip).
    \item Convergence rate (Cg rate).
    \item Need to compute $\mathrm{Prox}_f(\cdot)$ ($\mathrm{Prox}_f(\cdot)$).
\end{itemize}
\begin{table}[htp]
\caption{Summary of convergence results of SVRG-PDFP, SVRG-ADMM, SPDHG}
\centering
\begin{tabular}{|c|c|c|c|c|c|c|c|c|c|c|c|c|c|}
\hline
 \multicolumn{2}{|c|}{Algorithms}
&\multicolumn{1}{|c|}{Cg rate}
& \multicolumn{2}{|c|}{S.C. $f_(x)$}
& \multicolumn{2}{|c|}{S.C. $g^*(x)$}
& \multicolumn{2}{|c|}{FrkB}
& \multicolumn{2}{|c|}{Lip}
& \multicolumn{2}{|c|}{$\mathrm{Prox}_f(\cdot)$}
 \\
\hline
\multicolumn{2}{|c|}{SVRG-PDFP}
&\multirow{3}*{Linear} 
& \multicolumn{2}{|c|}{$\checkmark$}
& \multicolumn{2}{|c|}{$\checkmark$}
& \multicolumn{2}{|c|}{$-$}
& \multicolumn{2}{|c|}{$\checkmark$}
& \multicolumn{2}{|c|}{$-$}
\\
 \multicolumn{2}{|c|}{SVRG-ADMM}&
& \multicolumn{2}{|c|}{$\checkmark$}
& \multicolumn{2}{|c|}{$-$}
& \multicolumn{2}{|c|}{$\checkmark$}
& \multicolumn{2}{|c|}{$\checkmark$}
& \multicolumn{2}{|c|}{$-$}\\
\multicolumn{2}{|c|}{SPDHG}&
& \multicolumn{2}{|c|}{$\checkmark$}
& \multicolumn{2}{|c|}{$\checkmark$}
& \multicolumn{2}{|c|}{$-$}
& \multicolumn{2}{|c|}{$-$}
& \multicolumn{2}{|c|}{$\checkmark$}
\\
\hline
 \multicolumn{2}{|c|}{SVRG-PDFP}
&\multirow{3}*{$\mathcal{O}(1/k)$} 
& \multicolumn{2}{|c|}{$-$}
& \multicolumn{2}{|c|}{$-$}
& \multicolumn{2}{|c|}{$-$}
& \multicolumn{2}{|c|}{$\checkmark$}
& \multicolumn{2}{|c|}{$-$}\\
\multicolumn{2}{|c|}{SVRG-ADMM}&
& \multicolumn{2}{|c|}{$-$}
& \multicolumn{2}{|c|}{$-$}
& \multicolumn{2}{|c|}{$-$}
& \multicolumn{2}{|c|}{$\checkmark$}
& \multicolumn{2}{|c|}{$-$}\\
 \multicolumn{2}{|c|}{SPDHG}&
& \multicolumn{2}{|c|}{$-$}
& \multicolumn{2}{|c|}{$-$}
& \multicolumn{2}{|c|}{$-$}
& \multicolumn{2}{|c|}{$-$}
& \multicolumn{2}{|c|}{$\checkmark$}\\
\hline
\end{tabular}
\label{cgsum}
\end{table}

\section{Numerical Experiments.}
In this section, we show the numerical performance of the proposed algorithm. First we consider the graph guide logistic regression model \cite{Kim} and  compare SVRG-PDFP with other stochastic ADMM-type algorithms on two data sets. Then we present the results with  TV-$L_2$ model for image reconstruction of 2D and 3D images.
\subsection{Graph Guide Logistic Regression.}
The graph guide logistic regression model \cite{Kim} considers the following minimization problem
\begin{equation}\label{GGLGT1}
\underset{x \in \mathbb{R}^d}{\min}\frac{1}{n}\sum_{i = 1}^{n}f_i(x) + \nu_1\lVert x \rVert_2^2+ \nu_2\lVert Bx \rVert_1,  
\end{equation}
where $f_i(x) = \log(1 + \exp(-b_ia_i^Tx)),i = 1,2,\cdots,n$ and $b_i \in \{-1, 1 \}$ is the label of the sample $a_i \in \mathbb{R}^d$. As in \cite{STOCADMM,SVRGADMM,SCASADMM,SAGADMM}, we explore the graphical structure of the samples to prevent overfitting and use sparse inverse covariance selection \cite{Banerjee2008model} for the graph matrix $G$ and let  $B = [G;I]$.
The following are the details of the experiments:
\begin{itemize}
  \item Two data sets \textbf{a9a}: (54 features and 581012 samples) and \textbf{covtype} (123 features and 32561 samples) from \textbf{LIBSVM} \cite{CC01a} are used.  Half of the data are used for training and the other half for testing.
  \item A true solution  of (\textbf{\ref{GGLGT1}}) is obtain by running PDFP for $10000$ iterations so that the convergence of the problem is observed.
  \item We compare SVRG-PDFP with SCAS-ADMM \cite{SCASADMM}, OPG-ADMM \cite{RDAOPGADMM},  PDFP \cite{PDFP}, SPDFP \cite{ZZ20} and SVRG-ADMM \cite{SVRGADMM}. 
  \item For data set \textbf{a9a} the batch size is set as $b = 200$ for SPDFP, SCAS-ADMM, OPG-ADMM and $b = 20$ for SVRG-ADMM and SVRG-PDFP. For data set \textbf{covtype} the batch size is $b = 1000$ for SPDFP, SCAS-ADMM, OPG-ADMM and $b = 100$ for SVRG-ADMM and SVRG-PDFP. The batch size is chosen for the  algorithms  to have the best performance in terms of time. 
  \item The experiment is terminated when the relative error to the true solution is less than $1e^{-4}$ or reach the maximum iteration number.
  \item All the algorithms were run 10 times and the averaged is reported.
\end{itemize}


\begin{figure}[htp]
\begin{center}
\subfigure[a9a]{
\centering    
\includegraphics[width=0.45\columnwidth]{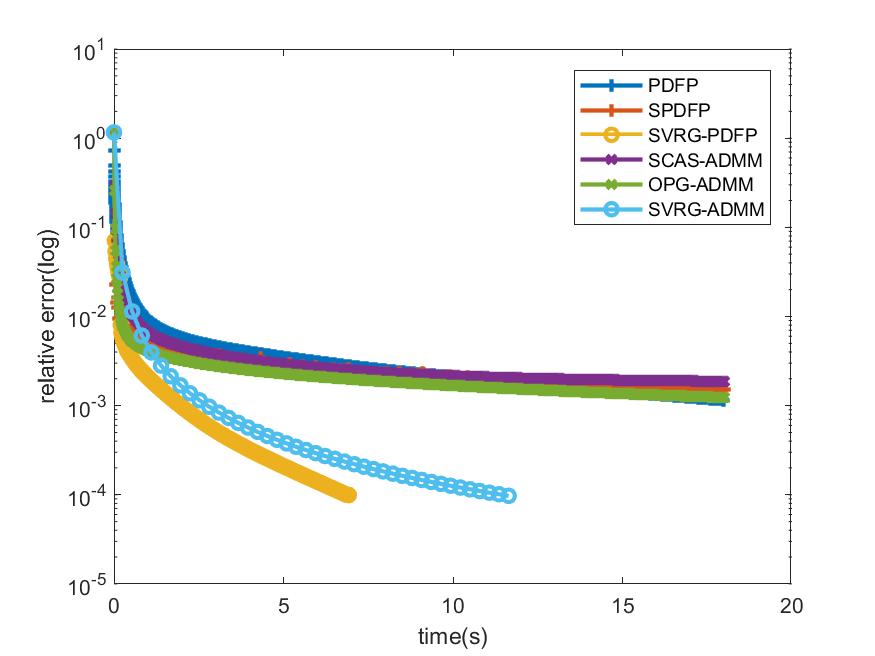}           
} 
\subfigure[covtype]{
\centering    
\includegraphics[width=0.45\columnwidth]{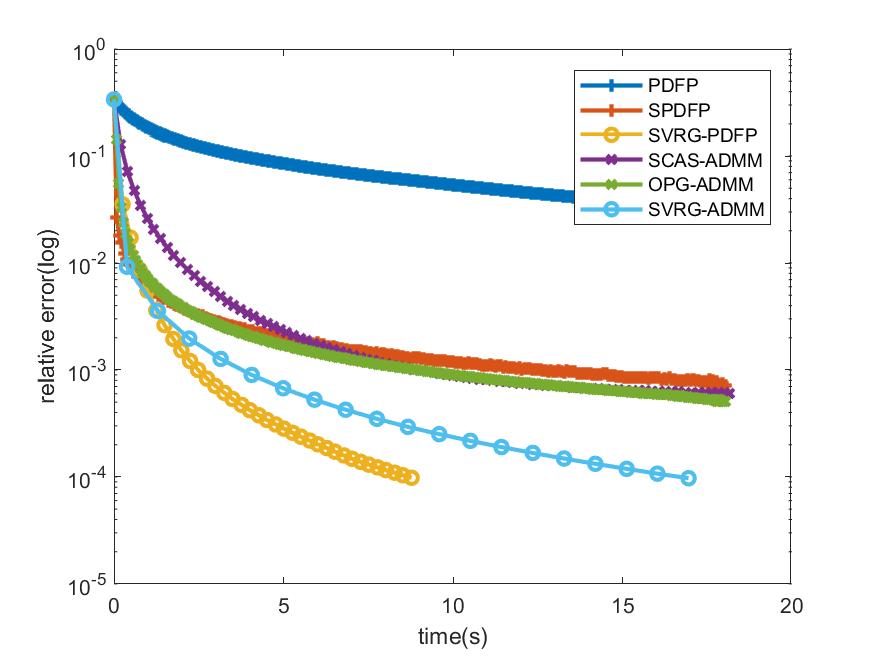}          
} 
\end{center}
\caption{Averaged relative error of objective value \textcolor{red}{vs.} time(s) over 10 independent repetitions.}
\label{GGlgt1}
\end{figure}

\begin{figure}[htp]
\begin{center}
\subfigure[a9a]{
\centering    
\includegraphics[width=0.45\columnwidth]{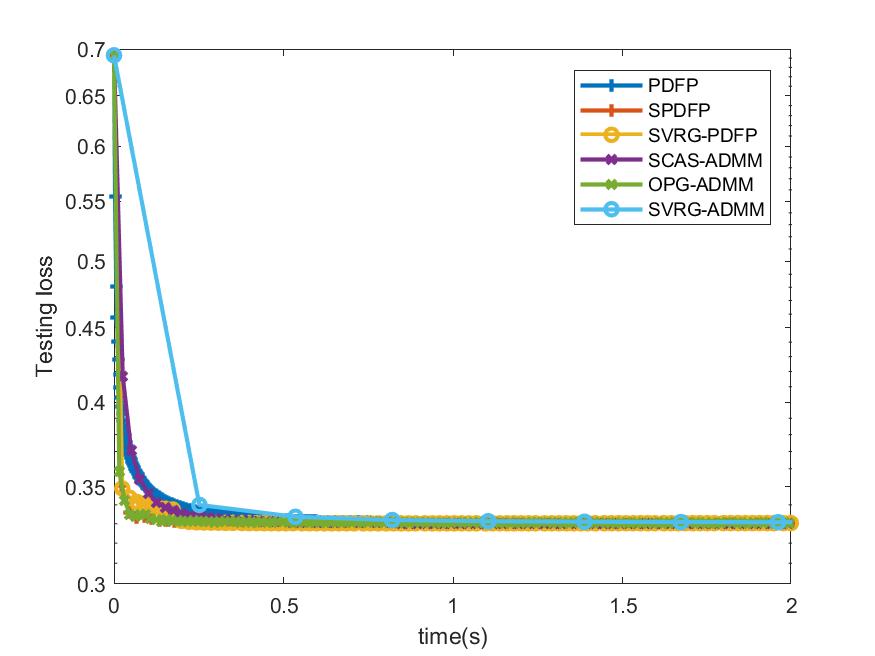}           
} 
\subfigure[covtype]{
\centering    
\includegraphics[width=0.45\columnwidth]{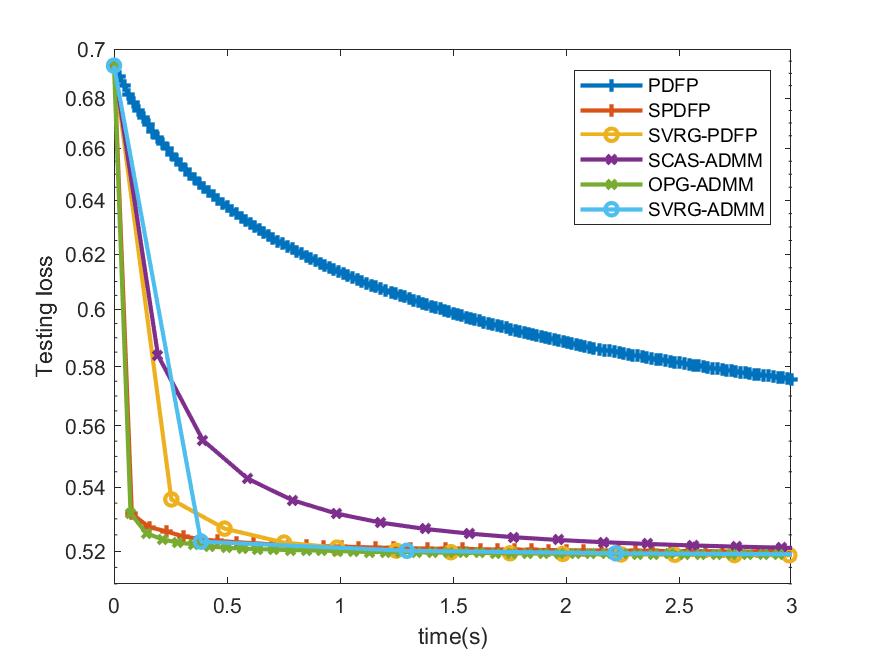}          
} 
\end{center}
\caption{Averaged testing loss \textcolor{red}{vs.} time(s) over 10 independent repetitions.}
\label{GGlgt2}
\end{figure}
 FIG \textbf{\ref{GGlgt1}}-\textbf{\ref{GGlgt2}} give the relative error to the minimum objective value on the  training sample and the testing loss of the two data set over time respectively. It can be seen that stochastic algorithms are generally better than deterministic algorithms (see FIG 1 (b)). Both SVRG-PDFP and SVRG-ADMM achieve a high accuracy solution faster than the other stochastic algorithm as SVRG based algorithms allow to use a constant step size while the other stochastic algorithms uses a diminishing step size. Compared with SVRG-ADMM, SVRG-PDFP performs slightly better on the relative error.  In term of testing loss, SVRG-PDFP is comparable to ADMM type algorithms. 



\subsection{Computerized tomography reconstruction.}
In this subsection we consider the computerized tomography reconstruction (CT) using TV-$L_2$ model i.e.
\begin{equation}\label{CTeq1}
\underset{x \in \mathbb{R}^d}{\arg\min}~\lVert \mathcal{A}x - f \rVert_2^2 + \nu\lVert \nabla x \rVert_1.     
\end{equation} 
Here $x$ is a vectorized image i.e. an $s_1 \times s_2$ 2D image is stacked into a $d = s_1 \times s_2$ dimensional column vector (3D case is  similar). The operator $\mathcal{A}$ is the X-ray transform, $f \in \mathbb{R}^n$ is the measured projections, $\nu > 0$ is a regularization parameter,  and $\nabla $ is the discrete gradient operator. The dimension of the operator $\mathcal{A}$ is generally very large, so traditionally we can use parallelization to compute the gradient of $\lVert \mathcal{A}x - f \rVert_2^2$ \cite{Xray}. We use this example to verify if stochastic algorithms can further reduce the computation cost.

\subsubsection{2D case.}
The follows are the settings of the experiment in 2D:
\begin{itemize}
    \item For operator $\mathcal{A}$, we use fan beam scanning geometry \cite{Xray} where the number of detectors is $nd = 512$, the number of viewers $nv = 360$. Thus the dimension of $f$ is $n = nd*nv = 184320$. 
    \item White noise with mean 0 variance 0.1 is added to the measured projection $f$.
    \item The proposed  SVRG-PDFP algorithm is compared with PDFP \cite{PDFP}, Stochastic PDFP (SPDFP) without SVRG \cite{ZZ20}, OPG-ADMM \cite{RDAOPGADMM}, and  SVRG-ADMM \cite{SVRGADMM}.
    \item The number of viewers is divided into $\frac{360}{nvb}$ non-overlap blocks(the number of viewers in each block is $nvb$). This yields that the batch size is $b = nvb * nd$. We choose $nvb = 20$ for OPG-ADMM, SPDFP and SCAS-ADMM and $nvb = 15$ for both SVRG-ADMM and SVRG-PDFP. 
    \item All the algorithms are terminated when they reach the maximum epoch number(effective pass).
    \item The experiment is performed on two different devices: NVIDIA GeForce GTX 1050 Ti GPU with 768 Cuda cores and TITAN RTX GPU with 4608 cores. The version of Matlab is 2018b. The comparisons of different algorithms on different devices will be reported. 
\end{itemize}  


\begin{figure}[htbp]
\centering
\subfigure[PSNR (TITAN RTX)]{
\centering
\includegraphics[width = 2.3 in]{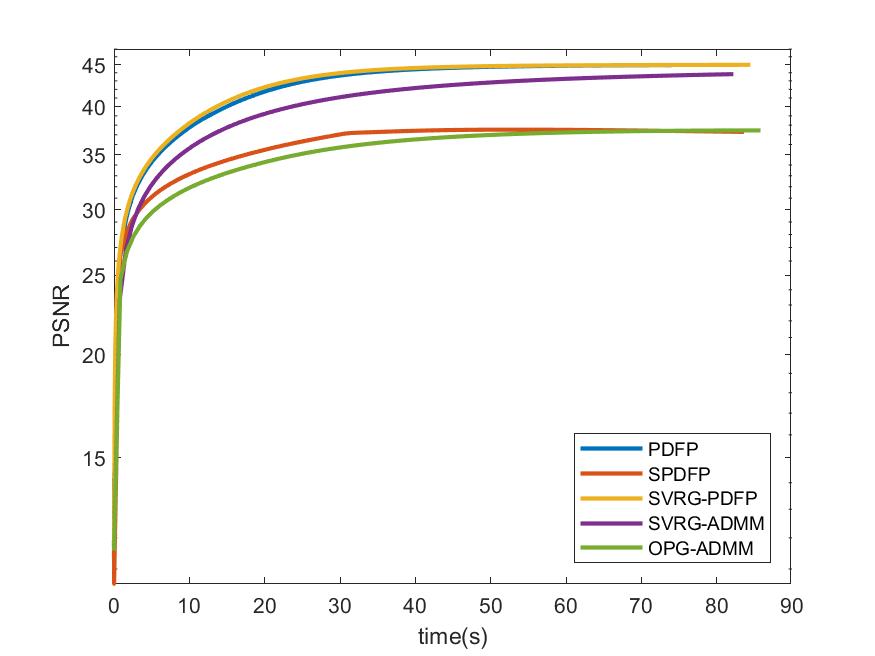}
}\hspace{1mm}
\subfigure[PSNR (GTX 1050 Ti)]{
\centering
\includegraphics[width = 2.3 in]{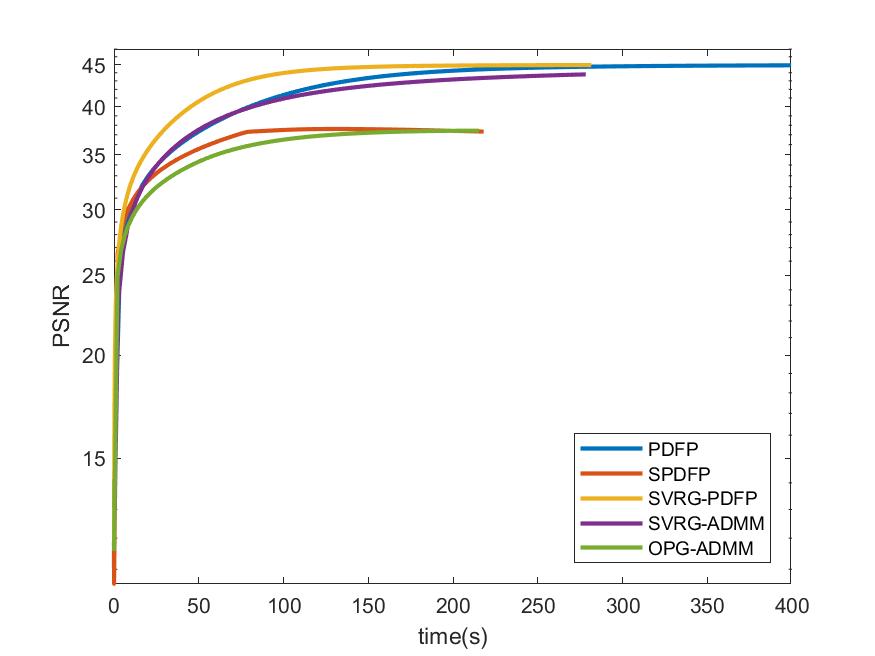}
}\hspace{1mm}
\centering
\caption{Different method for reconstructing image over 10 repetitions.}
\label{CT2-1}
\end{figure}

\begin{figure}[htbp]
\centering
\subfigure[PSNR = 30]{
\centering
\includegraphics[width = 2.2 in]{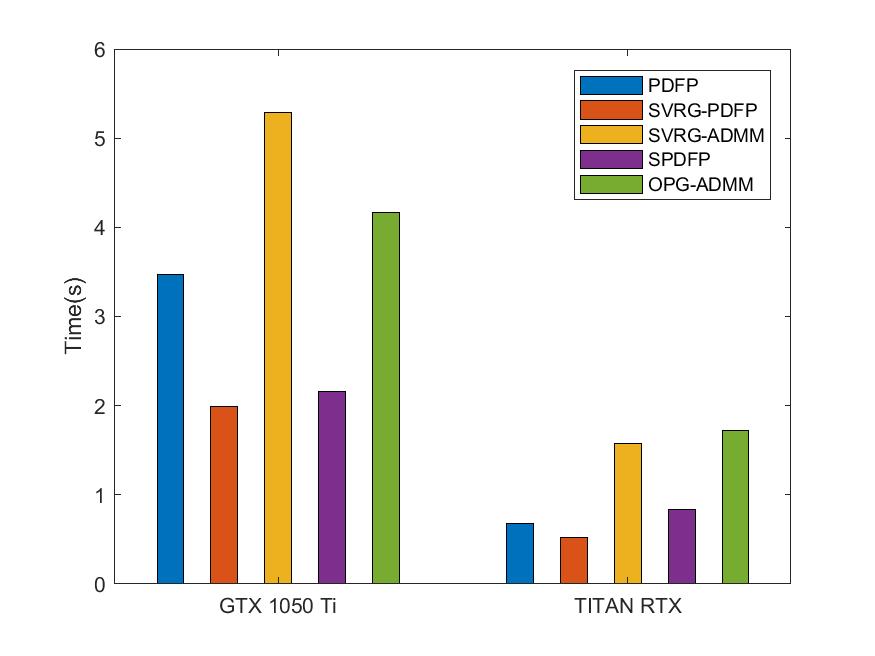}
}
\subfigure[PSNR = 35]{
\centering
\includegraphics[width = 2.2 in]{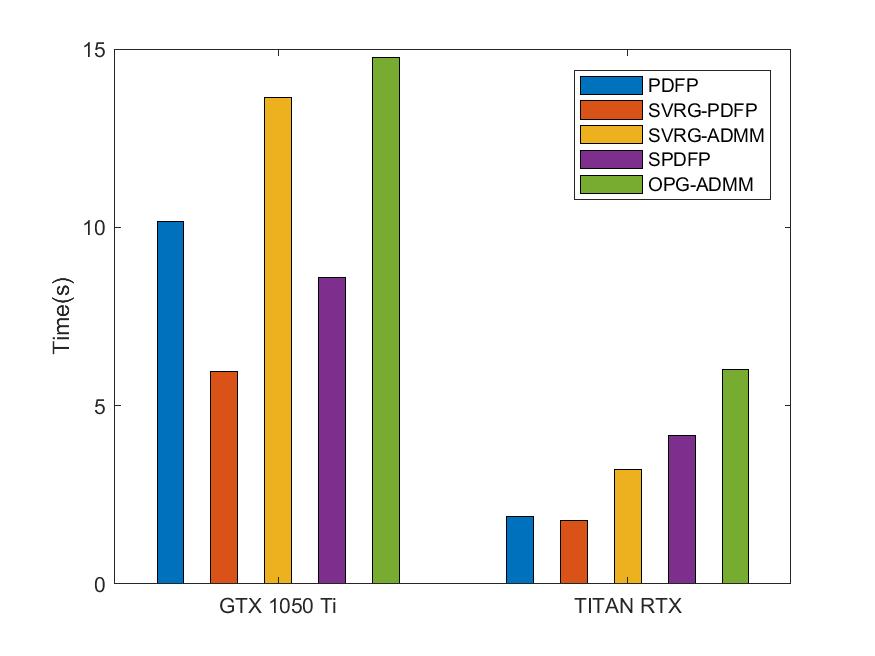}
}
\subfigure[PSNR = 37]{
\centering
\includegraphics[width = 2.2 in]{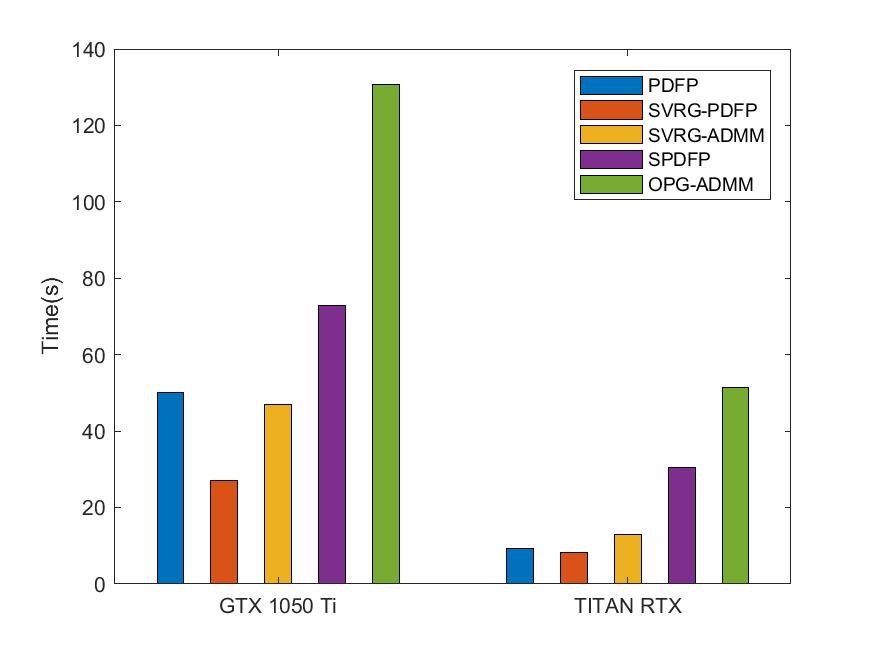}
}
\subfigure[PSNR = 43]{
\centering
\includegraphics[width = 2.2 in]{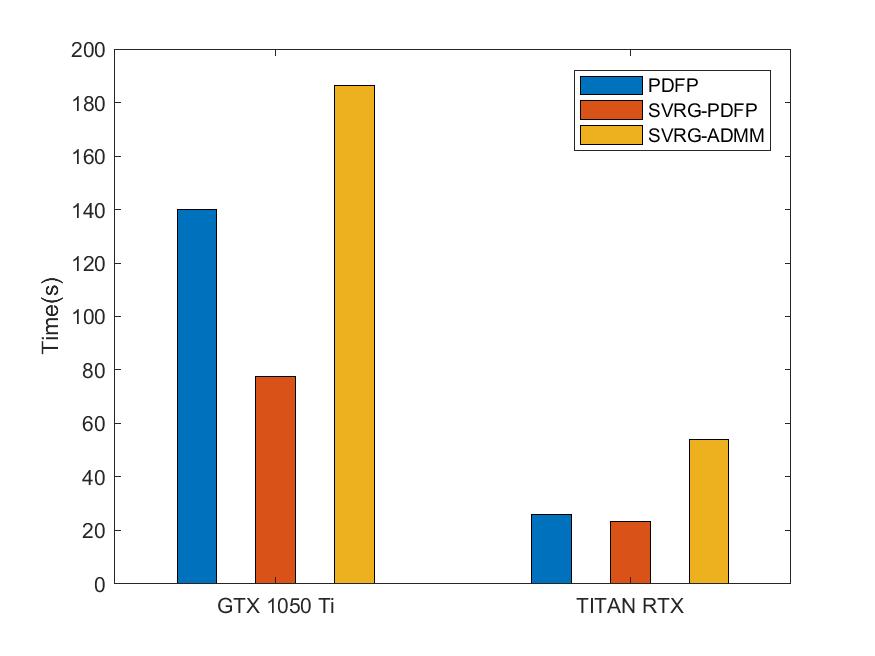}
}
\centering
\caption{Different method for reconstructing image over 10 repetitions.}
\label{CT2-11}
\end{figure}

\begin{figure}[htbp]
\centering
\subfigure[Ground truth]{
\centering
\includegraphics[width = 1.4 in]{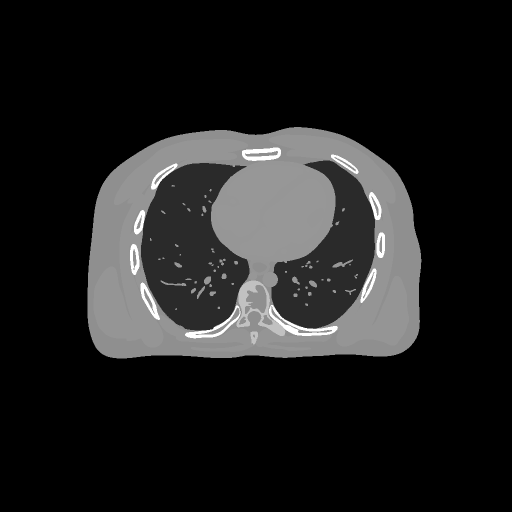}
}\hspace{1mm}
\subfigure[PDFP, PSNR= 44.94]{
\centering
\includegraphics[width = 1.4 in]{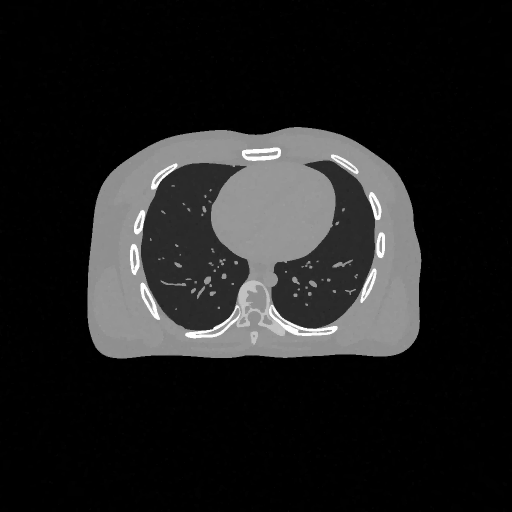}
}\hspace{1mm}
\subfigure[SPDFP, PSNR= 37.29]{
\centering
\includegraphics[width = 1.4 in]{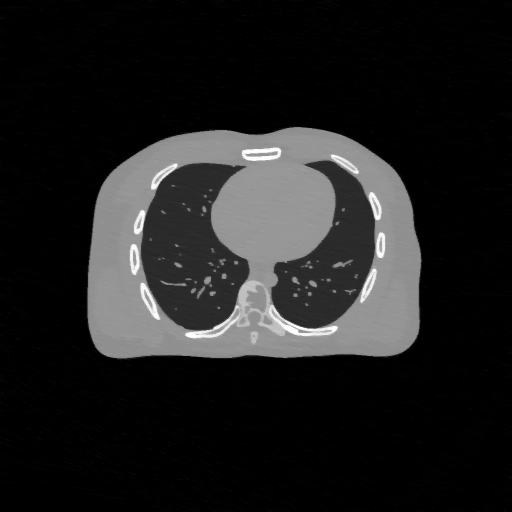}
}
\subfigure[SVRG-PDFP, PSNR= 44.98]{
\centering
\includegraphics[width = 1.4 in]{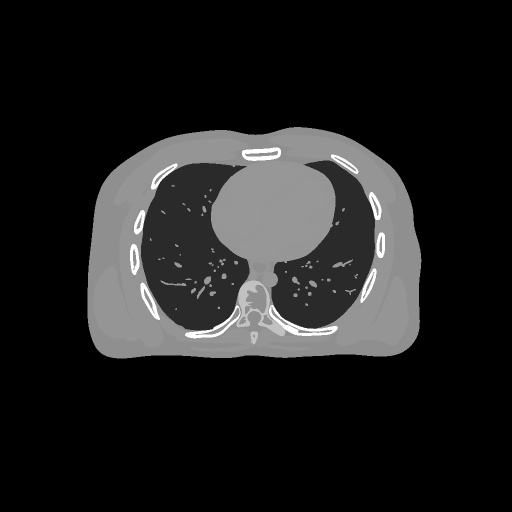}
}\hspace{1mm}
\subfigure[OPG-ADMM, PSNR= 37.45 ]{
\centering
\includegraphics[width = 1.4 in]{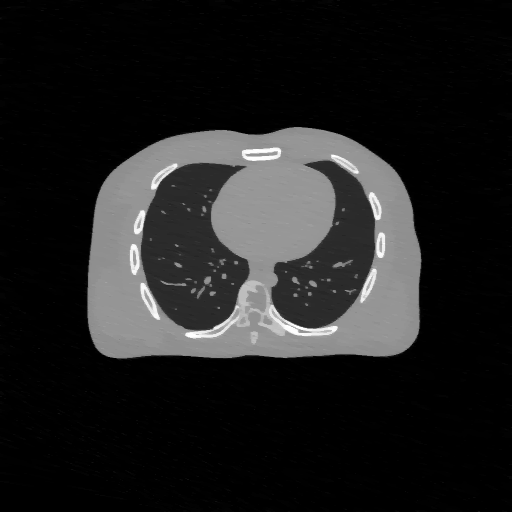}
}\hspace{1mm}
\subfigure[SVRG-ADMM, PSNR= 43.82]{
\centering
\includegraphics[width = 1.4 in]{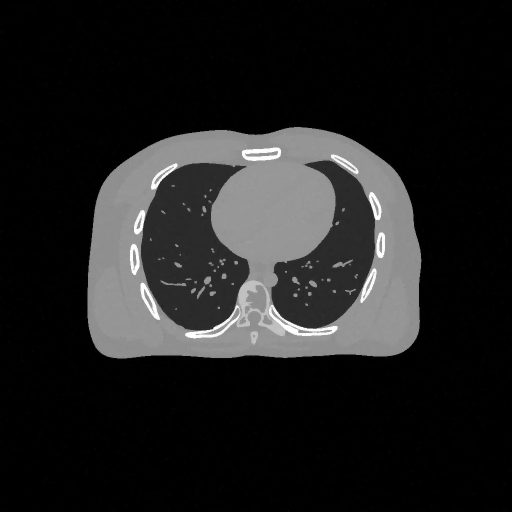}
}
\centering
\caption{ Reconstructed image (averaged) with different methods over 10 repetitions.}
\label{CT2-2}
\end{figure}

FIG \textbf{\ref{CT2-1}} give the results of the peak signal-to-noise ratio (PSNR) of the reconstructed images over time on two devices. It can be seen that stochastic algorithms without SVRG can not get high  PSNR comapred to  that with SVRG and the full batch PDFP.  The performance of SVRG-PDFP is as good as  PDPF with the device TITAN RTX while SVRG-PDFP behaves the best with  the devices with less cores. FIG \textbf{\ref{CT2-11}} record the computational time of different algorithms when PSNRs reach $30,35,37,43$ on the two different devices. We can see that when the computational resource is powerful (with many parallel cores), the full-batch PDFP can be highly parallized and the stochastic algorithm does not gain in general. However, when the cores number is not very high,  stochastic algorithms with SVRG are beneficial compared to deterministic algorithms.
FIG \textbf{\ref{CT2-2}} gives the reconstructed images with different algorithms and we can see that the one with SVRG-PDFP achieves the highest PSNR as the full batch PDFP.
\subsubsection{3D case.}
Here we also consider the 3D case as the number of unknowns and data are considerably  larger than 2D case. The follows are some difference to the settings of 3D case:
\begin{itemize}
\item {The size of image is $256*256*64$.}
    \item For the operator $\mathcal{A}$, we use cone beam scanning geometry \cite{Xray} where the parameter of detectors plane is $na \times nb = 512 \times 384$, and the number of viewers $nv = 668$. Thus the dimension of $f$ is $n = na*nb*nv = 131334144$ which is much larger than 2D case.
    \item The number of viewers is divided into $\frac{360}{nvb}$ non-overlap blocks and the number of viewers in each block is $nvb$. We set $nvb = 4$ for all the algorithms, i.e.  the batch size is $b = na * nb * nvb = 786432$.
   \end{itemize}

\begin{figure}[htbp]
\centering
\subfigure[PSNR (TITAN RTX)]{
\centering
\includegraphics[width = 2.3 in]{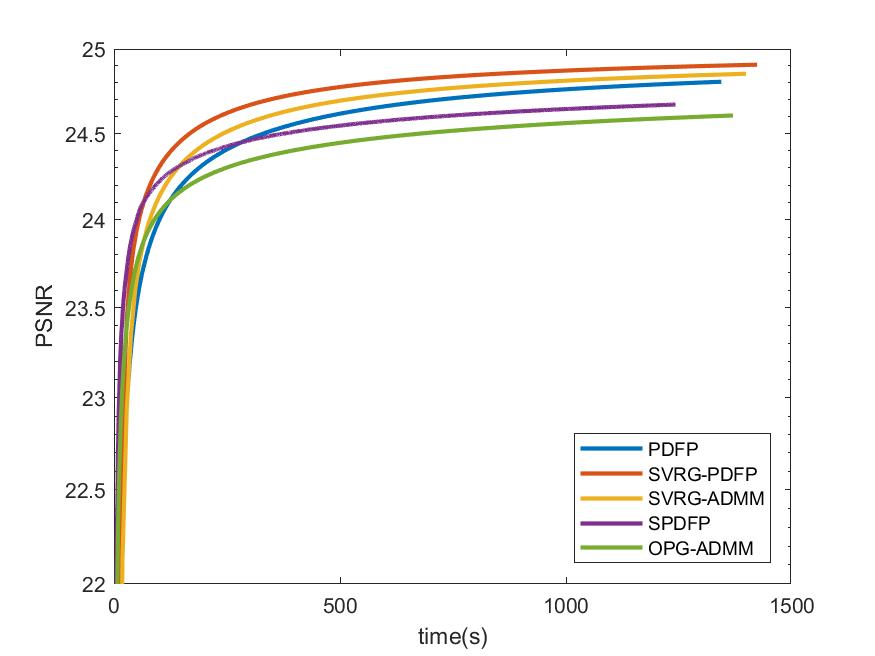}
}
\subfigure[PSNR (GTX 1050 Ti)]{
\centering
\includegraphics[width = 2.3 in]{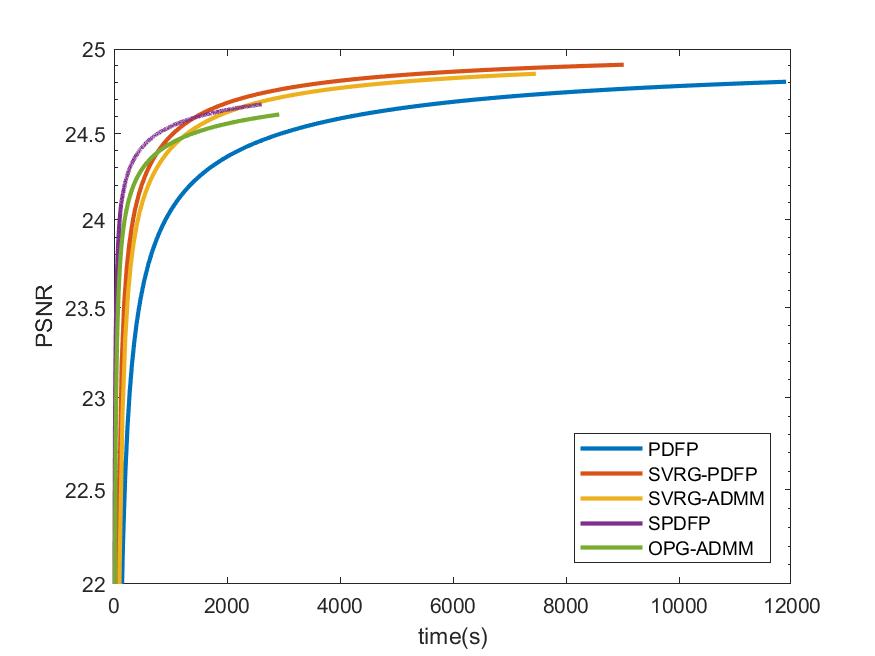}
}
\centering
\caption{PSNR for Different methods over 10 repetitions.}
\label{CT3-1}
\end{figure}

\begin{figure}[htbp]	
\centering
\subfigure[PSNR = 22]{
\centering
\includegraphics[width = 2.2 in]{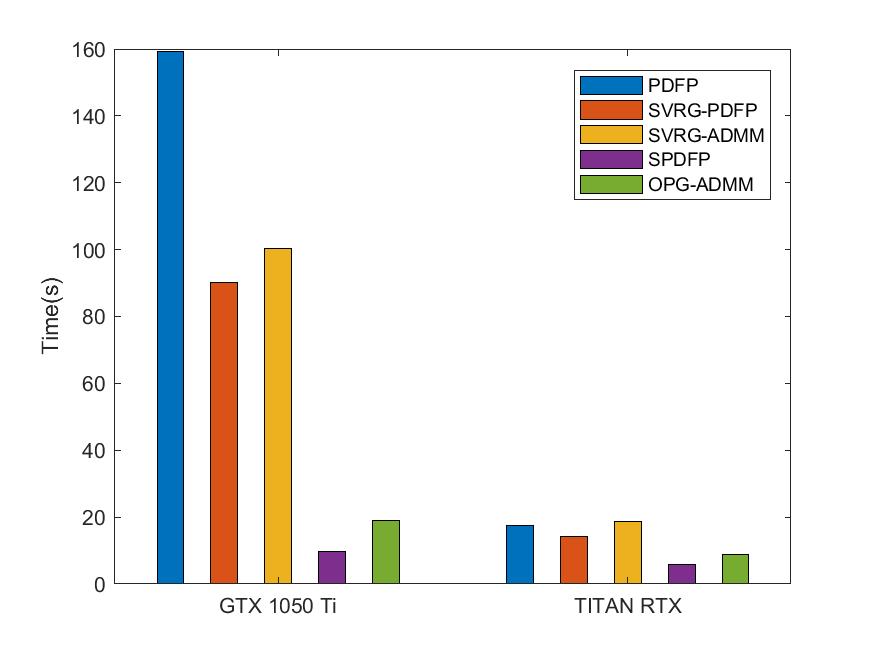}
}\hspace{1mm}
\subfigure[PSNR = 24]{
\centering
\includegraphics[width = 2.2 in]{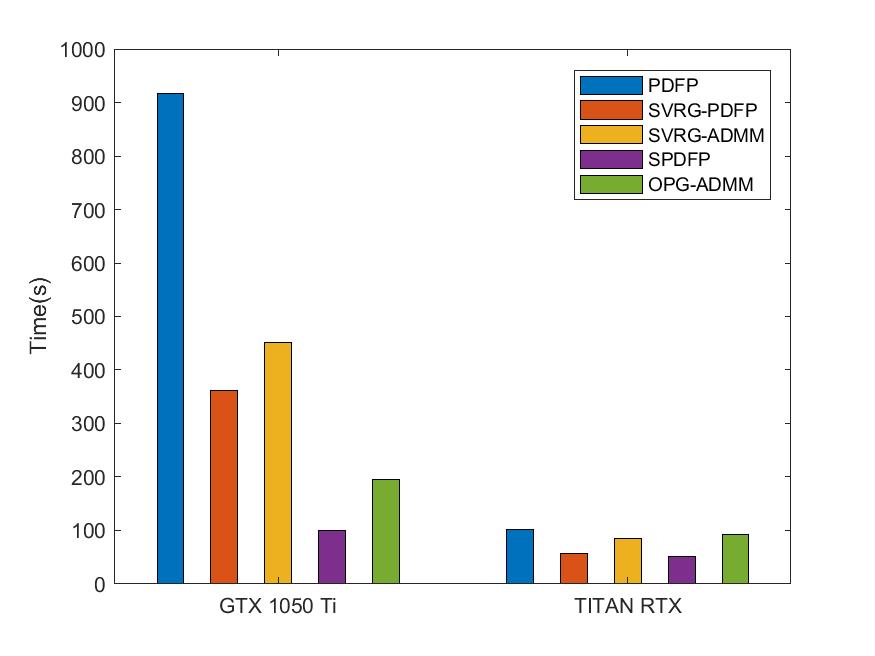}
}
\subfigure[PSNR = 24.6]{
\centering
\includegraphics[width = 2.2 in]{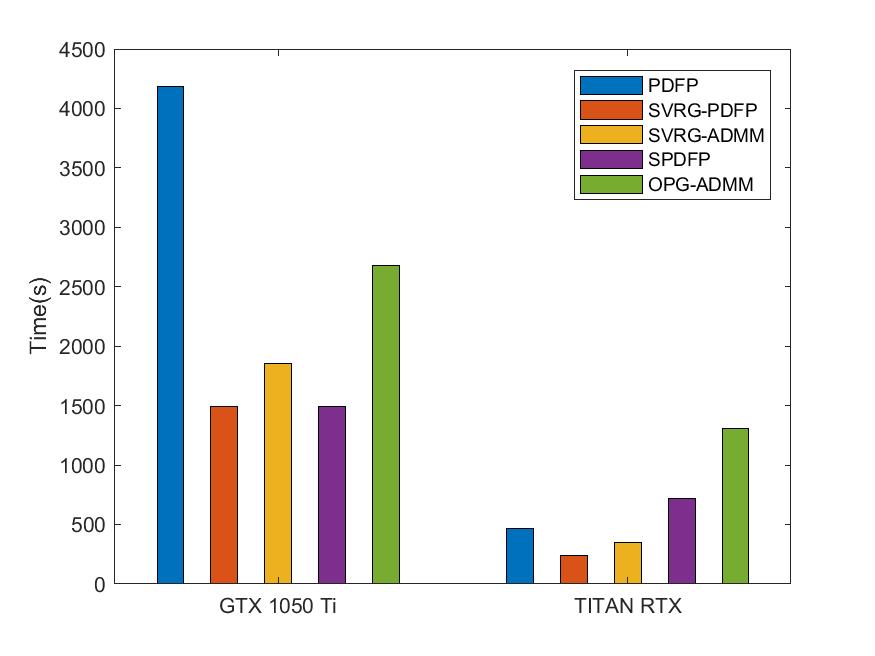}
}\hspace{1mm}
\subfigure[PSNR = 24.87]{
\centering
\includegraphics[width = 2.2 in]{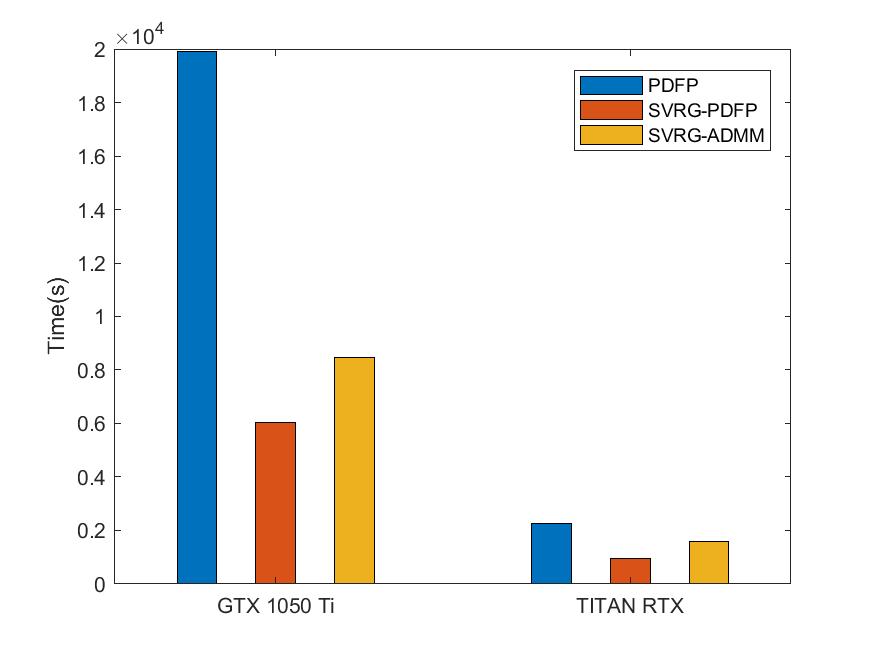}
}
\centering
\caption{Computation times of different methods for a given PSNR level over 10 repetitions.}
\label{CT3-11}
\end{figure}

\begin{figure}[htbp]
\centering
\subfigure[Ground truth]{
\centering
\includegraphics[width = 1.4 in ]{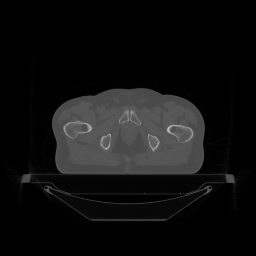}
}\hspace{1mm}
\subfigure[PDFP, PSNR= 50.59]{
\centering
\includegraphics[width = 1.4 in]{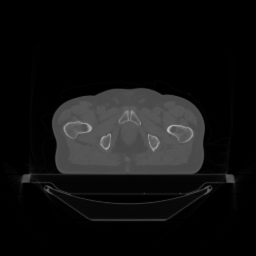}
}\hspace{1mm}
\subfigure[SPDFP, PSNR= 40.27]{
\centering
\includegraphics[width = 1.4 in]{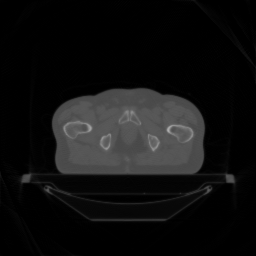}
}

\subfigure[SVRG-PDFP, PSNR= 51.79]{
\centering
\includegraphics[width = 1.4 in]{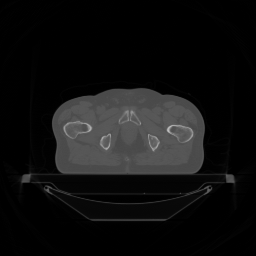}
}\hspace{1mm}
\subfigure[OPG-ADMM, PSNR= 41.24]{
\centering
\includegraphics[width = 1.4 in]{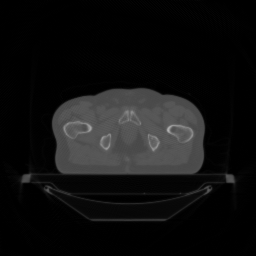}
}\hspace{1mm}
\subfigure[SVRG-ADMM, PSNR= 51.75]{
\centering
\includegraphics[width = 1.4 in]{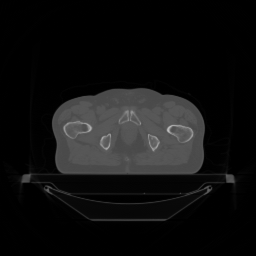}
}

\centering
\caption{Average of one slice of 3D image reconstruction of different method over 10 repetitions.}
\label{CT3-2}
\end{figure}
 FIG. \textbf{\ref{CT3-1}} give the results of PSNR of the images over time on two devices and FIG. \textbf{\ref{CT3-11}} show the computation time for different algorithms to achieve a given PSNR level (if achievable). It can be seen that the stochastic algorithms are generally quicker than deterministic algorithms as the problem size of this example is much larger than 2D case. The stochastic algorithms with SVRG perform better with both GPU devices in terms of both time and accuracy.  Finally, a slice of the reconstructed 3D images with different algorithms are shown  in FIG. \textbf{\ref{CT3-2}} to further verify the image quality of the reconstructed images of different algorithms.

\section{Discussions and Conclusion.}
In this paper, we proposed the stochastic variance reduced gradient primal dual fixed point method (SVRG-PDFP). We established the convergence rates O(1/k) and linear   for general and strongly convex cases respectively, which are standard results  for SVRG types of methods in the literature.  
Finally, numerical examples on both graph guide logistic regression and computed tomography reconstruction in 2D and 3D  are performed and compared to the full batch PDFP, stochastic PDFP (without SVRG) and the variants of stochastic ADMM. Our nuemrical results show that SVRG-PDFP show the advantages in terms of accuracy and computation speed, especially in the case of relatively limited parallel computing resource  in large scale problems. Thus the proposed algorithm could be useful for  CT reconstruction at clinics, where high performance computing resources are not at easy access.

\section*{Appendix.}
\textbf{Proof of Lemma \textbf{\ref{lm1}}:}
\begin{proof}
Let $\psi_{i_k} = \nabla f_{i_k}(x_k) - \nabla f_{i_k}(\tilde{x}) - (\nabla f(x_{k}) - \nabla f(\tilde{x}))$, one has
\begin{equation}\label{lm1eq1}
\begin{aligned}
\mathbb{E}\Big(\Big\lVert \frac{1}{b} \sum_{i_k \in I_k} \psi_{i_k} \Big\rVert^2\Big)
& = \frac{1}{b^2}\mathbb{E}\big(\sum_{i_k,i_{k'} \in I_k}\psi_{i_k}^T\psi_{i_{k'}}\big)	\\
& = \frac{1}{b^2}\mathbb{E}\Big(\sum_{i_k \not = i_{k'} \in I_k}\psi_{i_k}^T\psi_{i_{k'}}\Big)
+ \frac{1}{b}(\mathbb{E}\lVert \psi_i \rVert_2^2)	\\
& = \frac{b - 1}{bn(n - 1)}\mathbb{E}\Big(\sum_{i_k \not = i_{k'}}\psi_{i_k}^T\psi_{i_{k'}}\Big)
+ \frac{1}{b}\mathbb{E}(\lVert \psi_i \rVert_2^2)	\\
& = \frac{b - 1}{bn(n - 1)}\mathbb{E}\Big(\sum_{i_k, i_{k'} }\psi_{i_k}^T\psi_{i_{k'}}\Big)
- \frac{b - 1}{b(n - 1)}\mathbb{E}(\lVert \psi_i \rVert_2^2) + \frac{1}{b}\mathbb{E}(\lVert \psi_i \rVert_2^2)	\\
& = \frac{n - b}{b(n - 1)}\mathbb{E}(\lVert \psi_i \rVert_2^2), 
\end{aligned}
\end{equation}
where the last inequality follows from the fact that $\sum_{i = 1}^n \psi_{i} = 0$. 
Then 
\begin{equation}\label{lm1eq2}
\begin{aligned}
& \mathbb{E}\big(\lVert \nabla \hat{f}(x_k) - \nabla f(x_k) \rVert^2 \big) \\
& = \mathbb{E}\Big( \Big \lVert \frac{1}{b} \sum_{i_k \in I_k} (\nabla f_{i_k}(x_k) - \nabla f_{i_k}(\tilde{x}) - (\nabla f(x_{k}) - \nabla f(\tilde{x})) \Big \rVert^2 \Big) \\
& = \frac{n - b}{b(n - 1)}\mathbb{E}\big(\lVert \nabla f_{i_k}(x_k) - \nabla f_{i_k}(\tilde{x}) - (\nabla f(x_{k}) - \nabla f(\tilde{x})) \rVert^2 \big)\\
& = \frac{n - b}{b(n - 1)}\mathbb{E}\big(\lVert \nabla f_{i_k}(x_k) - \nabla f_{i_k}(\tilde{x})) \rVert^2 - \lVert  (\nabla f(x_{k}) - \nabla f(\tilde{x}) \rVert^2 \big) \\
& \leq \frac{n - b}{b(n - 1)}\mathbb{E}\big(\lVert \nabla f_{i_k}(x_k) - \nabla f_{i_k}(\tilde{x})) \rVert^2 \big) \\
& \leq \frac{2(n - b)}{b(n - 1)}\mathbb{E}\big(\lVert \nabla f_{i_k}(x_k) - \nabla f_{i_k}(x^*)) \rVert^2\big) + \frac{2(n - b)}{b(n - 1)}\mathbb{E}\big(\lVert \nabla f_{i_k}(\tilde{x}) - \nabla f_{i_k}(x^*)) \rVert^2\big) \\
& = \frac{2(n - b)}{b(n - 1)}\sum_{i = 1}^n \frac{1}{n}\lVert \nabla f_{i_k}(x_k) - \nabla f_{i_k}(x^*)) \rVert^2 + \frac{2(n - b)}{b(n - 1)}\sum_{i = 1}^n \frac{1}{n}\lVert \nabla f_{i_k}(\tilde{x}) - \nabla f_{i_k}(x^*)) \rVert^2\\
& \leq \frac{4L_{max}(n - b)}{b(n - 1)}\Big( f(x_k) - f(x^*) + f(\tilde{x}) - f(x^*) - \nabla f(x^*)^T(x_k + \tilde{x} - 2x^*) \Big) \\
& = 4L_{max}C(b)\big( D_f(x_k,x^*) + D_f(\tilde{x},x^*) \big).
\end{aligned}
\end{equation}
We note that the last inequality uses the fact: $\frac{1}{n} \sum_{i = 1}^n \lVert \nabla f_i(x) - \nabla f_i(x^*)\rVert_2^2 \leq 2L_{max}(f(x) - f(x^*) - \nabla f(x^*)^T(x - x^*))$ which can be found in Lemma 3.4 of \cite{SVRG}.
\end{proof}

\begin{lemma}\label{lm2}
Suppose $f(x)$ has $\frac{1}{\beta}$-Lipschitz continuous gradient, given $0 < \gamma \leq \beta$, the following estimate holds
\begin{equation}\label{lm2eq1}
\begin{aligned}
2\gamma(f(x_{k + 1}) - f(x)) 
& \leq \lVert x_{k} - x \rVert_2^2 - \lVert x_{k + 1} - x \rVert_2^2
+ 2\gamma^2\lVert \nabla \hat{f}(x_k) - \nabla f(x_k))\rVert_2^2 \\ 
& \qquad + 2\gamma (B(x - x_{k + 1}))^Tv_{k + 1} 
 - 2\gamma(\nabla \hat{f}(x_k) - \nabla f(x_k))^T(\overline{x}_k - x), 
\qquad \forall x \in \mathbb{R}^d
\end{aligned}
\end{equation}
where $x_{k}$ is the $kth$ inner iterate of Algorithm \textbf{$1$} and $\overline{x}_k = x_k - \gamma \nabla f(x_k) - \gamma B^Tv_{k + 1}$.
\end{lemma}
\begin{proof}
Recall the update of $x$ in Algorithm \textbf{$1$}, one has
\begin{equation}\label{formu1}
\nabla \hat{f}(x_k) + B^Tv_{k + 1} + \frac{1}{\gamma}(x_{k + 1} - x_{k}) = 0.   
\end{equation}
where 
\begin{equation}
\begin{aligned}
\nabla \hat{f}(x_k) & = \frac{1}{b}\sum_{i_k \in I_k}(\nabla f_{i_k}(x_k) - \nabla f_{i_k}(\tilde{x})) + \nabla f(\tilde{x}).\\
\end{aligned}   
\end{equation}
Using the convexity and $\frac{1}{\beta}$ Lipschitz continuous gradient of $f(x)$, we have
\begin{equation}\label{lm2eq2}
\begin{aligned}
f(x_{k + 1}) - f(x)
& = f(x_{k + 1}) - f(x_k) + f(x_k) - f(x) \\
& \leq \nabla f(x_k)^T (x_{k + 1} - x_k) + \frac{1}{2\beta}\lVert x_{k + 1} - x_k \rVert_2^2 + \nabla f(x_k)^T (x_k - x) \\
& = \nabla f(x_k)^T (x_{k + 1} - x) + \frac{1}{2\beta}\lVert x_{k + 1} - x_k \rVert_2^2.
\end{aligned}    
\end{equation}
Recall  Eq. \textbf{(\ref{formu1})}, one gets
\begin{equation}\label{lm2eq3}
\begin{aligned}
(x - x_{k + 1})^T(\nabla \hat{f}(x_k) + B^Tv_{k + 1} + \frac{1}{\gamma}(x_{k + 1} - x_k)) = 0. 
\end{aligned}
\end{equation}
Combing Eq. \textbf{(\ref{lm2eq2})} and \textbf{(\ref{lm2eq3})} and  using the fact $0 < \gamma \leq\beta$, one has
\begin{equation}\label{lm2eq4} 
\begin{aligned}
& f(x_{k + 1}) - f(x) \\
& \leq (\nabla \hat{f}(x_k) - \nabla f(x_k))^T(x - x_{k + 1})
+ (B(x - x_{k + 1}))^Tv_{k + 1} + \frac{1}{2\beta}\lVert x_{k + 1} - x_k \rVert_2^2 \\
& \qquad + \frac{1}{\gamma}(x - x_{k + 1})^T(x_{k + 1} - x_k) \\
&  = (\nabla \hat{f}(x_k) - \nabla f(x_k))^T(x - x_{k + 1})
+ (B(x - x_{k + 1}))^Tv_{k + 1} + \frac{1}{2\beta}\lVert x_{k + 1} - x_k \rVert_2^2 \\
& \qquad + \frac{1}{2\gamma}(\lVert x_k - x \rVert_2^2 -\lVert x_{k + 1} - x \rVert_2^2 - \lVert x_{k + 1} - x_k \rVert_2^2)\\
&  \leq (\nabla \hat{f}(x_k) - \nabla f(x_k))^T(x - x_{k + 1})
+ (B(x - x_{k + 1}))^Tv_{k + 1} 
+ \frac{1}{2\gamma}(\lVert x_k - x \rVert_2^2 -\lVert x_{k + 1} - x \rVert_2^2).\\
\end{aligned}
\end{equation}
Let $\overline{x}_k = x_k - \gamma \nabla f(x_k) - \gamma B^Tv_{k + 1}$, then Eq. \textbf{(\ref{lm2eq4})} can be rewritten as 
\begin{equation}\label{lm2eq5}
\begin{aligned}   
& f(x_{k + 1}) - f(x) \\
& \leq (\nabla \hat{f}(x_k) - \nabla f(x_k))^T(x - x_{k + 1})
+ (B(x - x_{k + 1}))^Tv_{k + 1}  
+ \frac{1}{2\gamma}(\lVert x_k - x \rVert_2^2 -\lVert x_{k + 1} - x \rVert_2^2)\\
& = (\nabla \hat{f}(x_k) - \nabla f(x_k))^T(x  - \overline{x}_k + \overline{x}_k -  x_{k + 1})
+ (B(x - x_{k + 1}))^Tv_{k + 1} + \frac{1}{2\gamma}(\lVert x_k - x \rVert_2^2\\
&\qquad -\lVert x_{k + 1} - x \rVert_2^2)\\
& \leq (\nabla \hat{f}(x_k) - \nabla f(x_k))^T(x  - \overline{x}_k) +
\gamma \lVert \nabla \hat{f}(x_k) - \nabla f(x_k) \rVert_2^2
+ (B(x - x_{k + 1}))^Tv_{k + 1} + \frac{1}{2\gamma}(\lVert x_k - x \rVert_2^2\\
&\qquad -\lVert x_{k + 1} - x \rVert_2^2),\\
\end{aligned}
\end{equation}
Multiplying both sides of Eq. \textbf{(\ref{lm2eq5})} by $2\gamma$, we get the result.
\end{proof}

\begin{lemma}\label{lm3}
Given $0 < \lambda \leq 1/\rho_{max}(BB^T)$, the following estimate holds:
\begin{equation}\label{lm3eq}
g^*(v_{k + 1}) - g^*(v) 
\leq (Bx_{k + 1})^T(v_{k + 1} - v) + \lVert v - v_{k} \rVert_G^2 - \lVert v - v_{k + 1} \rVert_G^2, \qquad  \forall v \in V.
\end{equation}
where $G = \frac{\gamma}{2\lambda}(I - \lambda BB^T)$ and $\rho_{max}(BB^T)$ denotes the maximum eigenvalue of matrix $BB^T$.
\end{lemma}
\begin{proof}
From the update of $v_{k + 1}$ in Algorithm $1$, one has
\begin{equation}\label{lm3eq1}
\begin{aligned}
&(v - v_{k + 1})^T\big(\frac{\lambda}{\gamma}\partial g^*(v_{k + 1}) + v_{k + 1} - v_{k} -\frac{\lambda}{\gamma}By_{k + 1}\big) 
 \geq 0; \\
\Leftrightarrow \qquad &(v - v_{k + 1})^T\big(\frac{\lambda}{\gamma}\partial g^*(v_{k + 1}) + v_{k + 1} - v_{k} 
-\frac{\lambda}{\gamma}B(x_k - \gamma \nabla \hat{f}(x_k) - \gamma B^Tv_k\big) 
 \geq 0; \\
\Leftrightarrow \qquad &(v - v_{k + 1})^T\big(\partial g^*(v_{k + 1}) - Bx_{k + 1} + \frac{\gamma}{\lambda}(I - \lambda BB^T)(v_{k + 1} - v_k)\big)
 \geq 0; \\
\Leftrightarrow \qquad &(v - v_{k + 1})^T\big(\partial g^*(v_{k + 1}) - Bx_{k + 1} + 2G(v_{k + 1} - v_k)\big)
 \geq 0; \\
\end{aligned}
\end{equation}
where in the last inequality we use the notation $G = \frac{\gamma}{2\lambda}(I - \lambda BB^T)$. Since $0 < \lambda \leq 1/\rho(BB^T)$, it can be easily verified that $G$ is positive semi-definite.\\
Using the convexity of $g^*(x)$ in last inequality of \textbf{(\ref{lm3eq1})}, one gets
\begin{equation}\label{lm3eq2}
g^*(v) - g^*(v_{k + 1}) + (v - v_{k + 1})^T\big( - Bx_{k + 1} + 2G(v_{k + 1} - v_k)\big) \geq 0.	
\end{equation}
Rearrange both sides of Eq. \textbf{(\ref{lm3eq2})} and use the fact that $2a^TGb = \lVert a + b \rVert_G^2  - \lVert a \rVert_G^2 - \lVert b \rVert_G^2$, then
\begin{equation}\label{lm3eq3}
\begin{aligned}
g^*(v_{k + 1}) - g^*(v) 
& \leq (v - v_{k + 1})^T\Big( - Bx_{k + 1} + 2G(v_{k + 1} - v_k)\Big) \\
& = (Bx_{k + 1})^T(v_{k + 1} - v) + 2(v - v_{k + 1})^TG(v_{k + 1} - v_k) \\
& \leq (Bx_{k + 1})^T(v_{k + 1} - v) + \lVert v_{k} - v \rVert_G^2 - \lVert v_{k + 1} - v \rVert_G^2 - \lVert v_{k + 1} - v_{k} \rVert_G^2 \\
& \leq (Bx_{k + 1})^T(v_{k + 1} - v) + \lVert v_{k} - v \rVert_G^2 - \lVert v_{k + 1} - v \rVert_G^2.
\end{aligned}
\end{equation}
This completes the proof.
\end{proof}

\noindent\textbf{Proof of Theorem \ref{thm1}}:
\begin{proof}
Let $x = x^*$ in Lemma \textbf{\ref{lm2}}, one has 
\begin{equation}\label{thm1eq1}
\begin{aligned}
& 2\gamma(f(x_{k + 1}) - f(x^*)) - 2\gamma (B(x^* - x_{k + 1}))^Tv_{k + 1}\\
& \leq \lVert x_{k} - x^* \rVert_2^2 - \lVert x_{k + 1} - x^* \rVert_2^2
+ 2\gamma^2\lVert \nabla \hat{f}(x_k) - \nabla f(x_k))\rVert_2^2 -2\gamma(\nabla \hat{f}(x_k) - \nabla f(x_k))^T\\
& \qquad (\overline{x} - x^*). 
\end{aligned}	
\end{equation}
Denote $\mathcal{I}_{k}$ as the information up to $k$-th inner iteration of Algorithm $1$. Taking conditional expectation w.r.t $\mathcal{I}_{k}$ in Eq. \textbf{(\ref{thm1eq1})}, noting $\mathbb{E}(\nabla \hat{f}(x_k)|\mathcal{I}_{k}) = \nabla f(x_k)$, we then have
\begin{equation}\label{thm1eq2}
\begin{aligned}
& 2\gamma \mathbb{E}\big(f(x_{k + 1})) - f(x^*)   - (B(x^* - x_{k + 1}))^Tv_{k + 1} |\mathcal{I}_{k}\big)\\
& \leq \lVert x_{k} - x^* \rVert_2^2 - \mathbb{E}(\lVert x_{k + 1} - x^* \rVert_2^2|\mathcal{I}_{k})
+ 2\gamma^2\mathbb{E}\big(\lVert \nabla \hat{f}(x_k) - \nabla f(x_k)\rVert_2^2|\mathcal{I}_{k}\big) \\
& \leq \lVert x_{k} - x^* \rVert_2^2 -\mathbb{E}(\lVert x_{k + 1} - x^* \rVert_2^2|\mathcal{I}_{k}) + 8\gamma^2L_{max}C(b)\Big( D_f(x_k,x^*) + D_f(\tilde{x},x^*) \Big).
\end{aligned}	
\end{equation}
Taking expectation over $\mathcal{I}_{k}$ for $k = 0,1,2,\cdots,m - 1$ and let $M = 4 L_{max}C(b)$, one gets
\begin{equation}\label{thm1eq3}
\begin{aligned}
& 2\gamma\mathbb{E}\big(f(x_{k + 1}) - f(x^*) - (B(x^* - x_{k + 1}))^Tv_{k + 1}\big)\\
& \leq \mathbb{E}(\lVert x_{k} - x^* \rVert_2^2) - \mathbb{E}(\lVert x_{k + 1} - x^* \rVert_2^2)
 + 8\gamma^2L_{max}C(b)\mathbb{E}\big( D_f(x_k,x^*) \big) 
 + 8\gamma^2L_{max}C(b) D_f(\tilde{x},x^*)  \\
& = \mathbb{E}(\lVert x_{k} - x^* \rVert_2^2) - \mathbb{E}(\lVert x_{k + 1} - x^* \rVert_2^2) + 2\gamma^2M\mathbb{E}\big( D_f(x_k,x^*) \big)
  + 2\gamma^2M D_f(\tilde{x},x^*) . 
\end{aligned}	
\end{equation}
Consider the left-hand side of \textbf{(\ref{thm1eq3})}, use the optimality condition for $x$ in Eq. \textbf{(\ref{eq21})} i.e. $\nabla f(x^*) + B^Tv^* = 0$, then
\begin{equation}\label{thm1eq4}
\begin{aligned}
& 2\gamma\mathbb{E}\big(f(x_{k + 1}) - f(x^*) -  (B(x^* - x_{k + 1}))^Tv_{k + 1}\big) \\
& = 2\gamma\mathbb{E}\big(f(x_{k + 1}) - f(x^*) - \nabla f(x^*)^T(x_{k + 1} - x^*) - (B^Tv^*)^T(x_{k + 1} - x^*) -  (B(x^* - x_{k + 1}))^Tv_{k + 1}\big) \\
& = 2\gamma\mathbb{E}\big(D_f(x_{k + 1},x^*) - (B(x^* - x_{k + 1}))^T(v_{k + 1} - v^*)\big).
\end{aligned}	
\end{equation}
Combining the inequalities \textbf{(\ref{thm1eq3})} and \textbf{(\ref{thm1eq4})}, one obtains
\begin{equation}\label{thm1eq5}
\begin{aligned}
& 2\gamma\mathbb{E}\big(D_f(x_{k + 1},x^*) - (B(x^* - x_{k + 1}))^T(v_{k + 1} - v^*)\big)\\
& \leq \mathbb{E}(\lVert x_{k} - x^* \rVert_2^2) - \mathbb{E}(\lVert x_{k + 1} - x^* \rVert_2^2) 
 + 2\gamma^2M\mathbb{E}\big( D_f(x_k,x^*)\big) + 2\gamma^2M D_f(\tilde{x},x^*)  \\
\end{aligned}	
\end{equation}
Taking the sum of  the  inequality \textbf{(\ref{thm1eq5})} from $k = 0,\cdots,m - 1$ and use the fact $x_0 = \tilde{x} = x_s$, one gets
\begin{equation}\label{thm1eq6}
\begin{aligned}
& 2\gamma\big(1 - \gamma M\big)\sum_{k = 1}^{m}\mathbb{E}\big(D_f(x_k,x^*)\big) - 2\gamma\mathbb{E}\sum_{k = 1}^m(B(x^* - x_{k}))^T(v_{k} - v^*)\\
& \leq \lVert x_0 - x^* \rVert_2^2 - \mathbb{E}(\lVert x_{m} - x^* \rVert_2^2)
 + 2(m + 1)\gamma^2M D_f(\tilde{x},x^*) \\
 & \leq \lVert \tilde{x}_{s} - x^* \rVert_2^2
 + 2(m + 1)\gamma^2M D_f(\tilde{x}_s,x^*) \\
\end{aligned}	
\end{equation}
By the convexity of $f(x)$, we have $f(\frac{1}{m}\sum_{k = 1}^m x_k) \leq \frac{1}{m}\sum_{k = 1}^{m}f(x_k)$. Noting that $\tilde{x}_{s + 1} = \frac{1}{m}\sum_{k = 1}^m x_k $, then 
\begin{equation}\label{thm1eq7}
\begin{aligned}
& 2\gamma\big(1 - \gamma M\big)m\mathbb{E}\big(D_f(\tilde{x}_{s + 1},x^*)\big) - 2\gamma\mathbb{E}\sum_{k = 1}^m(B(x^* - x_k))^T(v_{k} - v^*)\\
& \leq \lVert \tilde{x}_{s} - x^* \rVert_2^2
 + 2(m + 1)\gamma^2M D_f(\tilde{x}_{s},x^*) \\
\end{aligned}	
\end{equation}
Recall $R(x,v) = D_f(x,x^*) + D_{g^*}(v,v^*)$. Using Lemma \textbf{\ref{lm3}} and the convexity of $g^*$ and $\tilde{v}_{s + 1} = \frac{1}{m}\sum_{k = 1}^m v_k$, one has 
\begin{equation}\label{thm1eq8}
\begin{aligned}
& 2\gamma(1 - \gamma M)m\mathbb{E}R(\tilde{x}_{s + 1},\tilde{v}_{s + 1}) \\
& \leq 2\gamma(1 - \gamma M)m\mathbb{E}\big(D_f(\tilde{x}_{s + 1},x^*)) + 
2\gamma\mathbb{E}\sum_{k = 1}^{m}D_{g^*}(v_k,v^*) \\
& \leq 2\gamma(1 - \gamma M)m\mathbb{E}\big(D_f(\tilde{x}_s,x^*) \big) 
+  2\gamma\mathbb{E}\sum_{k = 1}^{m} (B(x_k - x^*))^T(v_{k} - v^*) \\
&\qquad + 2\gamma \lVert v_0 - v^* \rVert_G^2
 - 2\gamma \mathbb{E}(\lVert v_m - v^* \rVert_G^2   \\
& \leq \lVert \tilde{x}_{s} - x^* \rVert_2^2
 + 2(m + 1)\gamma^2M D_f(\tilde{x}_{s},x^*) + 2\gamma \lVert \tilde{v}_{s} - v^* \rVert_G^2 \\
& \leq \Big(\frac{2}{\mu_f} + 2(m + 1)\gamma^2M\Big) D_f(\tilde{x}_{s},x^*)  + 2\gamma \lVert \tilde{v}_{s} - v^* \rVert_G^2 \\
& \leq \Big(\frac{2}{\mu_f} + 2(m + 1)\gamma^2M\Big) D_f(\tilde{x}_{s},x^*)   + \frac{\gamma^2(1 - \rho_{min}(BB^T))}{\lambda} \lVert \tilde{v}_{s} - v^* \rVert_2^2 \\
& \leq \Big(\frac{2}{\mu_f} + 2(m + 1)\gamma^2M\Big) D_f(\tilde{x}_{s},x^*) + \frac{2\gamma^2(1 - \rho_{min}(BB^T))}{\lambda\mu_{g^*}} D_{g^*}(\tilde{v}_{s},v^*) \\
& \leq \Big\{ \frac{2}{\mu_f} + 2(m + 1)\gamma^2M + \frac{2\gamma^2(1 - \rho_{min}(BB^T))}{\lambda\mu_{g^*}} \Big\}R(\tilde{x}_{s},\tilde{v}_{s}) \\
& = \kappa R(\tilde{x}_{s},\tilde{v}_{s}),
\end{aligned}	
\end{equation}
where the first inequality is obtained by  $2\gamma\big(1 - \gamma M \big) \leq 2\gamma$ (since $\gamma \leq \frac{1}{M}$), the second inequality is from Lemma \textbf{\ref{lm3}} and the  third inequality is followed from  the inequality \textbf{(\ref{thm1eq7})}. The remain inequality follows from the strong convexity of $f$ and $g^*$.\\
Taking expectation of all the history, we have
\begin{equation}\label{thm1eq9}
\mathbb{E}(R(\tilde{x}_{s + 1},\tilde{v}_{s + 1})) \leq \kappa \mathbb{E}(R(\tilde{x}_{s},\tilde{v}_{s})),
\end{equation}
where 
\begin{equation*}
\kappa = \frac{1}{\mu_f\gamma(1 - \gamma M)m} + \frac{(m + 1)\gamma M}{(1 - \gamma M)m} + \frac{\gamma(1 - \rho_{min}(BB^T))}{\lambda\mu_{g^*}(1 - \gamma M)m}.
\end{equation*}
Therefore it yields 
\begin{equation}\label{thm1eq10}
\mathbb{E}R(\tilde{x}_s,\tilde{v}_s) \leq \kappa^s R(\tilde{x}_0,\tilde{v}_0).
\end{equation}
\end{proof}
\textbf{Proof of Corollary \textbf{\ref{thm11}}}:
\begin{proof}
It is easy to verify that $G = 0$, thus the term $\tilde{v}_{s} - v^* $ vanishes in  Eq. \textbf{(\ref{thm1eq8})}.
\end{proof}
\noindent \textbf{Proof of Theorem \textbf{\ref{thm2}}}:
\begin{proof}
Recall Eq. \textbf{(\ref{thm1eq5})} which {reads}
\begin{equation}\label{thm2eq1}
\begin{aligned}
& 2\gamma\mathbb{E}\Big(D_f(x_{k + 1},x^*) - (B(x^* - x_{k + 1}))^T(v_{k + 1} - v^*)\Big)\\
& \leq \mathbb{E}(\lVert x_{k} - x^* \rVert_2^2) - \mathbb{E}(\lVert x_{k + 1} - x^* \rVert_2^2) 
 + 2\gamma^2M\mathbb{E}\big( D_f(x_{k},x^*) \big) + 2\gamma^2M D_f(\tilde{x},x^*) \\
\end{aligned}	
\end{equation}
Sum Eq. \textbf{(\ref{thm2eq1})} from $k = 0,1,\cdots,m - 1$, then
\begin{equation}\label{thm2eq2}
\begin{aligned}
& 2\gamma\big(1 - \gamma M\big)\sum_{k = 1}^m\mathbb{E}\big(D_f(x_{k},x^*)\big) - 2\gamma\mathbb{E}\sum_{k = 1}^m(B(x^* - x_k))^T(v_k - v^*)\\
& \leq 2\gamma^2MD_f(x_{0},x^*) + \lVert x_0 - x^*\rVert_2^2 - \big(2\gamma^2M\mathbb{E}\big( D_f(x_{m},x^*)\big) + \lVert x_m - x^*\rVert_2^2\big)+ 2\gamma^2M mD_f(\tilde{x},x^*) \\
\end{aligned}	
\end{equation}
Let $\tilde{x}_{s + 1} = \frac{1}{m}\sum_{i = 1}^m x_k$, recall $\hat{x}_{s + 1} = x_m,x_0 = \hat{x}_{s}$, using the convexity of $f$, we have
\begin{equation}\label{thm2eq22}
\begin{aligned}
& 2\gamma\big(1 - \gamma M\big)m\mathbb{E}\big(D_f(\tilde{x}_s,x^*)\big) - 2\gamma\mathbb{E}\sum_{k = 1}^m(B(x^* - x_k))^T(v_k - v^*)\\
& \leq 2\gamma^2M D_f(\hat{x}_{s},x^*) + \lVert \hat{x}_{s} - x^*\rVert_2^2 - \big(2\gamma^2M\mathbb{E}\big(D_f(\hat{x}_{s + 1},x^*)\big) + \lVert \hat{x}_{s + 1} - x^*\rVert_2^2\big) \\
& \quad + 2\gamma^2M mD_f(\tilde{x}_{s},x^*).
\end{aligned}	
\end{equation}
By $\tilde{v}_{s + 1} = \frac{1}{m}\sum_{k = 1}^{m}v_k$,  the definition of  $R(\tilde{x},\tilde{v})$ and the fact that $\gamma \leq \frac{1}{2M}$, one gets 
\begin{equation}\label{thm2eq24}
\begin{aligned}
& 2\gamma(1 - 2\gamma M)m\mathbb{E}R(\tilde{x}_{s + 1},\tilde{v}_{s + 1}) \\
& \leq 2\gamma\big(1 - 2\gamma M\big)m\mathbb{E}\big(D_f(\tilde{x}_{s + 1},x^*)\big)
+ 2\gamma\mathbb{E}\sum_{k = 1}^{m} D_{g^*}(v_k,v^*)  \\
& \leq 2\gamma^2 MD_f(\hat{x}_{s},x^*) + \lVert \hat{x}_{s} - x^*\rVert_2^2 - \big(2\gamma^2 M \mathbb{E}\big(D_f(\hat{x}_{s + 1},x^*)\big) + \lVert \hat{x}_{s + 1} - x^*\rVert_2^2\big) \\
& \quad+ 2\gamma^2M m\big(D_f(\tilde{x}_{s},x^*)\big) - 2\gamma^2Mm\mathbb{E}\big(D_f(\tilde{x}_{s + 1},x^*)\big) + 2\gamma\lVert \hat{v}_{s} - v^*\rVert_2^2 - 2\gamma\mathbb{E}(\lVert \hat{v}_{s + 1} - v^*\rVert_2^2).
\end{aligned}   
\end{equation}
Denote $T_s = 2\gamma^2 M \mathbb{E}\big( D_f(\hat{x}_{s},x^*)\big) + \lVert \hat{x}_s - x^*\rVert_2^2 + 2\gamma^2M m\mathbb{E}\big(D_f(\tilde{x}_{s},x^*)\big)$. Taking the expectation of all the history, we have
\begin{equation}\label{thm2eq3}
\begin{aligned}
& 2\gamma(1 - 2\gamma M)m\mathbb{E}(R(\tilde{x}_{s + 1},\tilde{v}_{s + 1})) \\
& \leq T_{s} - T_{s + 1} + 2\gamma\mathbb{E}(\lVert \hat{v}_{s} - v^*\rVert_G^2) - 2\gamma\mathbb{E}(\lVert \hat{v}_{s + 1} - v^*\rVert_G^2).
\end{aligned}	
\end{equation}
Sum the equation \textbf{(\ref{thm2eq3})} from $0$ to $T - 1$, denote $\overline{x}_T = \frac{1}{T}\sum_{s = 1}^T \tilde{x}_s, \overline{v}_T= \frac{1}{T}\sum_{s = 1}^T \tilde{v}_s$ , then we get
\begin{equation}\label{thm2eq5}
\begin{aligned}
 2\gamma&(1 - 2\gamma M)mT\mathbb{E}(R(\overline{x}_T,\overline{v}_T)) \\
& \leq 2\gamma(1 - 2\gamma M)m\sum_{s = 1}^T\mathbb{E}(R(\tilde{x}_s,\tilde{v}_s)) \\
& \leq \sum_{s = 1}^T \Big( T_{s - 1} - T_s + 2\gamma \big(\mathbb{E}(\lVert \hat{v}_{s - 1} - v^* \rVert_G^2) - \mathbb{E}(\lVert \hat{v}_{s} - v^* \rVert_G^2 \big)\big) \Big) \\
& \leq T_0 + 2\gamma \lVert \hat{v}_0 - v^* \rVert_G^2 \\
& = 2\gamma^2MD_f(\hat{x}_{0},x^*) + \lVert \hat{x}_0 - x^*\rVert_2^2 + 2\gamma^2M mC(b)D_f(\hat{x}_{0},x^*) + 2\gamma\lVert \hat{v}_0 - v^* \rVert_G^2  \\
& = 2\gamma^2ME.
\end{aligned}	
\end{equation}
where $E = D_f(\hat{x}_{0},x^*) + \frac{\lVert \hat{x}_0 - x^*\rVert_2^2}{2\gamma^2M} +  mC(b)D_f(\hat{x}_{0},x^*) + \frac{\lVert \hat{v}_0 - v^* \rVert_G^2}{\gamma M}$. This yields
\begin{equation}\label{thm2eq6}
\begin{aligned}
& \mathbb{E}(R(\overline{x}_T,\overline{v}_T)) 
\leq  \frac{\gamma ME}{(1 - 2\gamma M)mT}.
\end{aligned}	
\end{equation}

\end{proof}

\end{document}